\documentclass[10pt]{amsart}

\usepackage[all]{xy}
\usepackage{amssymb}
\usepackage{eucal}

\usepackage[utf8]{inputenc}
\usepackage[T1]{fontenc}
\usepackage[backend=biber,
            isbn=false,
            doi=false,
            maxbibnames=5,
            giveninits=true,
            style=alphabetic,
            citestyle=alphabetic]{biblatex}
\usepackage{url}
\renewbibmacro{in:}{}
\bibliography{geometricquantization.bib}

\usepackage{color}
\definecolor{e-mail}{rgb}{0,.40,.80}
\definecolor{reference}{rgb}{.20,.60,.22}
\definecolor{citation}{rgb}{0,.40,.80}

\usepackage[colorlinks=true,
            linkcolor=reference,
            citecolor=citation,
            urlcolor=e-mail]{hyperref}
\usepackage{breakurl}
\usepackage{cleveref}

\newcommand{\cA}{\mathcal{A}}
\newcommand{\cC}{\mathcal{C}}
\newcommand{\cD}{\mathcal{D}}
\newcommand{\cE}{\mathcal{E}}
\newcommand{\cG}{\mathcal{G}}
\newcommand{\cH}{\mathcal{H}}
\newcommand{\cK}{\mathcal{K}}
\newcommand{\cL}{\mathcal{L}}
\newcommand{\cM}{\mathcal{M}}
\newcommand{\cO}{\mathcal{O}}
\newcommand{\cQ}{\mathcal{Q}}
\newcommand{\cS}{\mathcal{S}}

\newcommand{\bA}{\mathbf{A}}
\newcommand{\C}{\mathbf{C}}
\newcommand{\bE}{\mathbb{E}}
\newcommand{\bL}{\mathbb{L}}
\newcommand{\bP}{\mathbb{P}}
\newcommand{\R}{\mathbf{R}}
\newcommand{\bT}{\mathbb{T}}

\newcommand{\Z}{\mathbf{Z}}

\newcommand{\g}{\mathfrak{g}}
\newcommand{\h}{\mathfrak{h}}
\newcommand{\fl}{\mathfrak{l}}
\newcommand{\fm}{\mathfrak{m}}
\newcommand{\fp}{\mathfrak{p}}
\newcommand{\fu}{\mathfrak{u}}

\newcommand{\B}{\mathrm{B}}
\renewcommand{\d}{\mathrm{d}}
\renewcommand{\H}{\mathrm{H}}
\newcommand{\N}{\mathrm{N}}
\newcommand{\T}{\mathrm{T}}

\newcommand{\BD}{\mathrm{BD}}
\newcommand{\Bun}{\mathrm{Bun}}
\newcommand{\CAlg}{\mathrm{CAlg}}

\newcommand{\Ch}{\mathrm{Ch}}

\newcommand{\cl}{\mathrm{cl}}
\newcommand{\coad}{\mathrm{coad}}
\newcommand{\Crit}{\mathrm{Crit}}
\newcommand{\CrysCat}{\mathrm{CrysCat}}
\newcommand{\curv}{\mathrm{curv}}
\newcommand{\ddr}{\mathrm{d}_{\mathrm{dR}}}
\newcommand{\dive}{\mathrm{div}}
\newcommand{\dlog}{\ddr\log}
\newcommand{\dR}{\mathrm{dR}}
\newcommand{\DR}{\mathrm{DR}}

\newcommand{\End}{\mathrm{End}}
\newcommand{\ev}{\mathrm{ev}}
\newcommand{\fib}{\mathrm{fib}}
\newcommand{\free}{\mathrm{free}}

\newcommand{\Gm}{\mathbf{G}_{\mathrm{m}}}
\newcommand{\hCat}{\widehat{\mathrm{Cat}}_\infty}
\newcommand{\ham}{/\!\!/}
\newcommand{\HH}{\mathrm{HH}}
\newcommand{\Hom}{\mathrm{Hom}}
\newcommand{\id}{\mathrm{id}}
\newcommand{\Ind}{\mathrm{Ind}}
\newcommand{\IndCoh}{\mathrm{IndCoh}}
\newcommand{\LagrCorr}{\mathrm{LagrCorr}}
\newcommand{\Lie}{\mathrm{Lie}}
\newcommand{\LieAlgbroid}{\mathrm{LieAlgbroid}}
\newcommand{\LocSys}{\mathrm{LocSys}}
\newcommand{\Map}{\mathrm{Map}}
\newcommand{\MF}{\mathrm{MF}}
\newcommand{\Mod}{\mathrm{Mod}}
\newcommand{\LMod}{\mathrm{LMod}}
\newcommand{\oblv}{\mathrm{oblv}}
\newcommand{\Obs}{\mathrm{Obs}}
\newcommand{\Pic}{\mathrm{Pic}}

\newcommand{\pol}{\mathrm{pol}}
\newcommand{\prequant}{\mathrm{prequant}}
\newcommand{\PrL}{\mathrm{Pr}^{\mathrm{L}}}
\newcommand{\pt}{\mathrm{pt}}
\newcommand{\QCoh}{\mathrm{QCoh}}
\newcommand{\RealSqZ}{\mathrm{RealSqZ}}
\newcommand{\Rep}{\mathrm{Rep}}
\newcommand{\ShvCat}{\mathrm{ShvCat}}
\newcommand{\SL}{\mathrm{SL}}
\renewcommand{\sl}{\mathfrak{sl}}
\newcommand{\Sym}{\mathrm{Sym}}
\newcommand{\trivial}{\mathrm{trivial}}
\newcommand{\Vect}{\mathrm{Vect}}

\newcommand{\llbr}{[\![}
\newcommand{\rrbr}{]\!]}

\DeclareMathOperator{\Spec}{Spec}

\newcommand{\defterm}[1]{\textbf{\emph{#1}}}

\newtheorem{thm}{Theorem}[section]
\newtheorem*{thmintro}{Theorem}
\newtheorem{prop}[thm]{Proposition}
\crefname{prop}{proposition}{propositions}
\newtheorem*{propintro}{Proposition}
\newtheorem{cor}[thm]{Corollary}
\newtheorem{lm}[thm]{Lemma}
\newtheorem{conjecture}[thm]{Conjecture}

\theoremstyle{definition}
\newtheorem{defn}[thm]{Definition}

\theoremstyle{remark}
\newtheorem{remark}[thm]{Remark}
\newtheorem{example}[thm]{Example}

\begin{document}
\title{Shifted geometric quantization}
\author{Pavel Safronov}
\address{School of Mathematics, University of Edinburgh, Edinburgh, UK}
\email{p.safronov@ed.ac.uk}
\begin{abstract}
We introduce geometric quantization in the setting of shifted symplectic structures. We define Lagrangian fibrations and prequantizations of shifted symplectic stacks and their geometric quantization. In addition, we study many examples including symplectic groupoids, Hamiltonian spaces and moduli spaces of flat connections. In the case of Hamiltonian spaces we prove a derived analog of the ``quantization commutes with reduction'' principle.
\end{abstract}
\maketitle

\section*{Introduction}

We define geometric quantizations of shifted symplectic stacks in the sense of \cite{PTVV}, so that a geometric quantization of an $n$-shifted symplectic structure gives rise to an $(\infty, n)$-category. Our main results are:
\begin{itemize}
\item We give a definition of a prequantization of $n$-shifted Lagrangian fibrations, the input data to geometric quantization (\cref{def:prequantumfibration}).

\item We give the definition of the geometric quantization of prequantum $n$-shifted Lagrangian fibrations (\cref{def:geometricquantization1stack}).

\item We prove that a derived intersection of $n$-shifted Lagrangians equipped with $n$-shifted Lagrangian fibrations carries an $(n-1)$-shifted Lagrangian fibration (\cref{thm:prequantumLagrangianfibrationintersection}).

\item Given a derived stack $X$ equipped with an $n$-gerbe $\cG$ with characteristic class $c_1(\cG)$, we prove that the $n$-shifted twisted cotangent bundle $\T^*_{c_1(\cG)}[n] X$ carries a natural prequantum $n$-shifted Lagrangian fibration (\cref{thm:prequantumtwistedcotangent}).

\item We give many examples of prequantum $n$-shifted Lagrangian fibrations: coming from symplectic groupoids, coadjoint orbits, Slodowy slices and moduli spaces of flat connections (\cref{sect:examples}).

\item Given a Hamiltonian $G$-scheme $X$ equipped with a $G$-equivariant prequantization and polarization, we prove that its derived Hamiltonian reduction carries an induced $0$-shifted Lagrangian fibration whose geometric quantization is the space of $G$-invariants in the geometric quantization of $X$ (the quantization commutes with reduction principle, \cref{prop:QR}).
\end{itemize}

\subsection*{Quantization}

Let us briefly recall some mathematical formalizations of the idea of a quantization of a physical system. Classical mechanics is specified by a symplectic manifold $X$ together with a Hamiltonian function $H\colon X\rightarrow \R$. The corresponding quantum-mechanical system is described, in particular, by the following data:
\begin{itemize}
\item The \emph{algebra of observables} $\Obs$ (usually taken to be a $C^*$-algebra, but we will consider it as a plain algebra) together with a (self-adjoint) element $H\in\Obs$.

\item The \emph{space of states} $\cH$, an $\Obs$-module (usually taken to be a Hilbert space, but, again, we will consider it as a plain vector space).
\end{itemize}

A long-standing question in mathematical physics is about the relationship between the two sets of data. There are the following mathematical procedures:
\begin{itemize}
\item (\emph{Deformation quantization}). The algebra of observables is supposed to be a family of algebras parametrized by the Planck's constant $\hbar$. In formal deformation quantization \cite{BFFLS} one considers an algebra $\Obs$ flat over $\C\llbr \hbar\rrbr$ together with an isomorphism $\Obs|_{\hbar = 0}\cong C^\infty(X; \C)$ and satisfying $(ab - ba)/\hbar = \{a, b\} \pmod{\hbar}$ for every $a,b\in \Obs$, where $\{-, -\}$ is the Poisson bracket on $C^\infty(X)$ induced by the symplectic structure.

\item (\emph{Geometric quantization}). The space of states $\cH$ is supposed to have, roughly speaking, half of the degrees of freedom in $C^\infty(X)$. In geometric quantization the space of states is the space of sections of an appropriate line bundle over $X$ satisfying constraints specified by a Lagrangian distribution.
\end{itemize}

Besides mechanics, quantization is also relevant for field theories. Namely, given an $n$-dimensional classical field theory, the phase space associated to a hypersurface in spacetime (a closed $(n-1)$-dimensional manifold) is a symplectic manifold, so one may consider its deformation and geometric quantizations. In topological field theories one may also define phase spaces associated to lower-dimensional manifolds. Following \cite{AKSZ,CMRClassical,CalaqueTFT}, we will equip the phase spaces with a \emph{shifted symplectic structure}, the notion introduced in \cite{PTVV}. Then the phase space of an $n$-dimensional classical field theory associated to a closed $d$-dimensional manifold carries an $(n-d-1)$-shifted symplectic structure. Note that it also makes sense for $d = n$ in which case the corresponding $(-1)$-shifted symplectic structure is related to the Batalin--Vilkovisky antibracket. As an example, the 3-dimensional classical Chern--Simons theory associated to a compact Lie group $K$ has the phase spaces given by the moduli spaces of flat $K$-connections:
\begin{itemize}
\item $M$ is a closed oriented 3-manifold. The phase space is given by the critical locus of the Chern--Simons functional, so it carries the induced Batalin--Vilkovisky $(-1)$-shifted symplectic structure.

\item $M$ is a closed oriented 2-manifold. The phase space is endowed with the Atiyah--Bott symplectic structure.

\item $M=S^1$. The phase space is the stack of conjugacy classes $[K/K]$. The corresponding $1$-shifted symplectic structure is encoded, in particular, by the canonical 3-form on $K$.

\item $M=\pt$. The phase space is the classifying stack $\B K=[\pt/K]$. The corresponding $2$-shifted symplectic structure is encoded in the invariant symmetric bilinear form on $\Lie(K)$ used in the definition of the Chern--Simons theory.
\end{itemize}

The theory of shifted symplectic structures is developed the most in the context of derived algebraic geometry \cite{ToenDAG}, so from now on all our spaces will be derived Artin stacks over the field $k$. The algebra of (polynomial) functions on $X$ is denoted by $\cO(X)$. Quantizations of $n$-shifted symplectic stacks look as follows.

\begin{itemize}
\item (\emph{Deformation quantization}). The theory of deformation quantization of $n$-shifted symplectic stacks has not been well-developed, but let us mention \cite{CPTVV,Pridham0,Pridham1}. The idea is that the algebra of functions $\cO(X)$ carries a canonical $\bP_{n+1}$-structure, i.e. a Poisson bracket of cohomological degree $(-n)$. Deformation quantization should be given, in particular, by an $\bE_{n+1}$-algebra, i.e. an algebra over the operad of little $(n+1)$-disks. In the classical case (i.e. $n=0$) $\bP_1$-algebras are ordinary Poisson algebras and $\bE_1$-algebras are associative algebras.

\item (\emph{Geometric quantization}). This is the theory developed in this paper. The geometric quantization of an $n$-shifted symplectic stack is an $n$-category $\cC$ (more precisely, we use the formalism of $(\infty, n)$-categories). Given an $\bE_{n+1}$-algebra $A$, one can make sense of an $A$-linearity structure on $\cC$ and this gives a higher version of the action of observables on state spaces. As an example, a dg category $\cC$ has the Hochschild cohomology $\HH^\bullet(\cC)$, an $\bE_2$-algebra, and an $A$-linear structure on $\cC$ is given by a morphism of $\bE_2$-algebras $A\rightarrow \HH^\bullet(\cC)$ (see \cref{rem:LinCat} for the general notion in a stable presentable setting).
\end{itemize}

Note that the above picture is supposed to make sense for $n=-1$ as well. In that case an $\bE_0$-algebra is simply a vector space (or a chain complex) with a distinguished vector. Indeed, global observables in a quantum field theory do not have any algebraic structure, but there is a distinguished trivial observable. A module over an $\bE_0$-algebra $A$ is given by a functional $A\rightarrow k$ and in a quantum field theory such a functional is given by the path integral.

Let us mention that there are other approaches for describing ``higher'' symplectic structures on phase spaces. For instance, there is a theory of multisymplectic geometry, where an $n$-shifted symplectic structure is replaced by an $(n+1)$-plectic structure \cite{RomanRoy,Rogers}, i.e. a closed $(n+2)$-form (of degree zero) satisfying a certain nondegeneracy condition. A closed $(n+2)$-form of degree zero is an example of a closed $2$-form of degree $n$, but the nondegeneracy conditions for $n$-plectic and $n$-shifted symplectic structures diverge for $n > 0$. In this setting the theory of geometric quantization has been developed in \cite{Rogers,FRS, Nuiten}.

\subsection*{Geometric quantization}

The theory of geometric quantization was developed by Souriau \cite{Souriau} and Kostant \cite{KostantGQ} and goes back to Kirillov's orbit method \cite{Kirillov}. Some of the textbook references include \cite{Woodhouse,BatesWeinstein,EEMCRRVM}.

Let us sketch the relevant definition in the algebro-geometric context (see \cite{BeilinsonKazhdan}). Let $(X, \omega)$ be a smooth symplectic scheme over a field $k$ of characteristic zero. Geometric quantization requires the following data:
\begin{itemize}
\item (Prequantization) A line bundle $(\cL, \nabla)$ with a connection on $X$ whose curvature is $\omega$.

\item (Polarization) A Lagrangian foliation $\cM\subset \T_X$: a subbundle closed under the Lie bracket which is Lagrangian with respect to $\omega$.
\end{itemize}

The output of the geometric quantization is the vector space $\Gamma_{\nabla}(X, \cL)\subset \Gamma(X, \cL)$ of sections of $\cL$ flat along $\cM$. Note that we do not include a metaplectic correction. One may reformulate this definition as follows. The closed two-form $\omega$ allows one to consider the category $\cD_\omega(X)$ of $\omega$-twisted $\cD$-modules on $X$; a particular class of objects is given by vector bundles with a connection with central curvature $\omega\cdot\id$. Then we have functors
\[\Vect\xrightarrow{\text{prequantization}} \cD_\omega(X)\xrightarrow{\text{polarization}} \Vect,\]
where the first functor sends $V\mapsto V\otimes \cL$ and the second functor computes flat sections of a $\cD$-module along $\cM$. The composite gives the space of states.

To categorify geometric quantization, let us recall that for a scheme (or a derived stack) $X$ one has the following associated linear objects:
\begin{itemize}
\item The commutative algebra (chain complex) of global functions $\cO(X)$. The version with flat connections is the algebra $\Omega^\bullet(X)$ of differential forms. Given an element $\alpha\in\H^0(X, \Omega^{\geq 1})$, a closed one-form of degree 0, we may consider the twisted version $\Omega^\bullet_\alpha(X)$ with the differential $\ddr + \alpha\wedge(-)$.

\item The symmetric monoidal $\infty$-category of quasi-coherent complexes $\QCoh(X)$. The version with flat connections is the $\infty$-category of $\cD$-modules $\cD(X)=\QCoh(X_{\dR})$ (where $X_{\dR}$ is the de Rham stack). Given an element $\alpha\in\H^1(X, \Omega^{\geq 1})$, a closed one-form of degree 1, we may consider the twisted version $\cD_\alpha(X)$.

\item The symmetric monoidal $(\infty, 2)$-category of quasi-coherent sheaves of categories $\ShvCat(X)$ (see \cite{GaitsgoryShvCat}, \cite[Chapter 10]{SAG}). The version with flat connections is the $(\infty, 2)$-category of crystals of categories $\CrysCat(X) = \ShvCat(X_{\dR})$. Given an element $\alpha\in\H^2(X, \Omega^{\geq 1})$, a closed one-form of degree 2, we may consider the twisted version $\CrysCat_\alpha(X)$.

\item ...
\end{itemize}

We review some of these notions in \cref{sect:QCohQStk}. To categorify geometric quantization, we consider a derived stack $X$ with a $1$-shifted symplectic structure $\omega\in\H^1(X, \Omega^{\geq 2})\subset \H^2(X, \Omega^{\geq 1})$. Then we should get
\[\PrL\xrightarrow{\text{prequantization}} \CrysCat_\omega(X)\xrightarrow{\text{polarization}}\PrL,\]
where $\PrL$ is the $\infty$-category of presentable $\infty$-categories. More concretely, a prequantization is given by $(\cG, \nabla)$, a gerbe with a connective structure on $X$ \cite{Brylinski}, whose obstruction to having a curving is given by $\curv(\cG, \nabla) = \omega$. One can again define a polarization to be a 1-shifted Lagrangian foliation (see \cite{ToenVezzosi,BSY} for what this means in the present context) and take flat sections along that. In the main body of the text we consider a simplified approach where we assume the gerbe descends along a morphism $X\rightarrow B$ (whose fibers give the Lagrangian foliation) and simply take global sections of the gerbe on $B$. Here by global sections of a gerbe $\cG$ on $B$ we mean the $\infty$-category $\QCoh^\cG(B)$ of twisted quasi-coherent sheaves. We refer to \cref{def:prequantumfibration} for the full definition of prequantization and polarization and to \cref{def:geometricquantization1stack} for the definition of geometric quantization of this data.

The above picture admits an obvious extension to higher categories. But it also admits an extension to lower categories. Namely, for a $(-1)$-shifted symplectic stack $(X, \omega)$ over the ground field $k$ one may contemplate geometric quantization in the form of maps
\[k\xrightarrow{\text{prequantization}} \Omega_\omega(X)\xrightarrow{\text{polarization}}k.\]
We do not comment on the notion of a $(-1)$-shifted polarization in this paper. But the notion of a $(-1)$-shifted prequantization turns out to be closely related to the notion of the Batalin--Vilkovisky (BV) quantization \cite{Schwarz}. We prove the following result (see \cref{prop:BVquantization}) closely related to the work of \v{S}evera \cite{Severa}.

\begin{propintro}
Let $Y$ be a smooth scheme and $\omega$ the standard $(-1)$-shifted symplectic structure on $\T^*[-1] Y$. Then one has a quasi-isomorphism
\[\Omega_\omega(\T^*[-1] Y)\cong (\Gamma(\T^*[-1] Y, \pi^* K_Y)[-\dim Y], \Delta),\]
where $\Delta$ is the BV operator.
\end{propintro}

The notion of $(-1)$-shifted geometric quantization should be compared with the structure of global observables in quantum field theories, which have maps
\[k\xrightarrow{\trivial}\Obs\xrightarrow{\text{path integral}} k\]
whose composite is the partition function of the theory.

So, one may consider the work \cite{Schwarz} as giving a BV analog of geometric quantization while the approach via $\BD_0$-algebras (Beilinson--Drinfeld algebras) \cite[Chapter 7.3]{CostelloGwilliam2} is giving a BV analog of deformation quantization.

Different categorical levels are intertwined in the classical setting by the notion of an $n$-shifted Lagrangian morphism $L\rightarrow X$ \cite{PTVV}, where $X$ is an $n$-shifted symplectic stack. Such a structure often appears when one studies classical field theories defined on manifolds with boundary \cite{CMRClassical}. After geometric quantization $X$ gives rise to an $(\infty, n)$-category and $L\rightarrow X$ gives rise to an object. In the case $n = 0$ we get a structure closely related to the one appearing in the quantum BV-BFV formalism \cite{CMRQuantum} (note, however, that the authors consider a somewhat restrictive class of $0$-shifted Lagrangians of the form $L'\times L''\rightarrow X$, where $X$ is a symplectic manifold, $L''\hookrightarrow X$ is a smooth Lagrangian submanifold and $L'$ is a $(-1)$-shifted symplectic manifold).

\subsection*{Examples}

Apart from giving the definitions, the bulk of the paper is concerned with providing examples of (and providing tools to construct) prequantum $n$-shifted Lagrangian fibrations.

Recall that an important way to construct $n$-shifted symplectic structures is as follows. Given an $n$-shifted symplectic stack $X$ and two Lagrangian morphisms $L_1, L_2\rightarrow X$, the derived intersection $L_1\times_X L_2$ carries an $(n-1)$-shifted Lagrangian structure \cite[Theorem 2.9]{PTVV}. We extend this result to Lagrangians which carry prequantum data (\cref{thm:prequantumLagrangianfibrationintersection}).

\begin{thmintro}
Consider the diagram
\[
\xymatrix@R=0cm{
L_1 \ar[dr] \ar[dd] && L_2 \ar[dl] \ar[dd] \\
& X \ar[dd] & \\
B_{L_1} \ar[dr] && B_{L_2} \ar[dl] \\
& B_X &
}
\]
of derived stacks. Suppose $X$ is $n$-shifted symplectic, $L_1, L_2\rightarrow X$ are Lagrangian morphisms, $X\rightarrow B_X$ is a Lagrangian fibration and $L_i\rightarrow B_{L_i}$ are relative Lagrangian fibrations. Then the intersection $L_1\times_X L_2$ carries a Lagrangian fibration $L_1\times_X L_2\rightarrow B_{L_1}\times_{B_X} B_{L_2}$. The same claim is true if the Lagrangians are equipped with prequantizations.
\end{thmintro}

Given any derived Artin stack $X$ locally of finite presentation, Calaque \cite{CalaqueCotangent} shows that the $n$-shifted cotangent bundle $\T^*[n] X$ possesses a canonical $n$-shifted symplectic structure. Given an element $\alpha\in\H^{n+1}(X, \Omega^{\geq 1})$, a closed one-form of degree $(n+1)$ on $X$, one may consider the twisted cotangent bundle $\T^*_\alpha[n] X$, see \cite[Section A.1]{BeilinsonKazhdan} for the definition in the underived context. Also, an $n$-gerbe $\cG$ on $X$ has a characteristic class $c_1(\cG)\in\H^{n+1}(X, \Omega^{\geq 1})$ obstructing the existence of a flat connection on $\cG$. Using the above statement about intersections of prequantized Lagrangians, we prove the following result (see \cref{thm:prequantumtwistedcotangent}).

\begin{thmintro}
Let $X$ be a derived Artin stack locally of finite presentation and $\cG$ an $n$-gerbe on $X$. Then
\[\T^*_{c_1(\cG)}[n] X\longrightarrow X\]
has a natural structure of a prequantum $n$-shifted Lagrangian fibration.
\end{thmintro}

Another interesting example of a shifted symplectic stack is the classifying stack $\B G$ of an affine algebraic group $G$ equipped with nondegenerate invariant symmetric bilinear pairing on its Lie algebra. In this case we can prove negative results (see \cref{prop:BGprequantization,prop:evendimension}).

\begin{propintro}
Let $G$ be a connected reductive group.
\begin{itemize}
\item $\B G$ does not admit a prequantization.
\item If $\dim(G)$ is odd, $\B G$ does not admit a Lagrangian fibration.
\end{itemize}
\end{propintro}

Note that the absence of a prequantization is a deficiency of the algebraic situation and in the case of compact Lie groups the corresponding prequantization (but not polarization) exists, see \cite{Waldorf}.

Let us list examples described in \cref{sect:examples}:
\begin{itemize}
\item Given a symplectic groupoid $\cG\rightrightarrows X$, the quotient stack $[X/\cG]$ carries a 1-shifted symplectic structure and the projection $X\rightarrow [X/\cG]$ a Lagrangian structure. We show that a prequantization and polarization of this Lagrangian gives rise to a multiplicative prequantization and polarization of the symplectic groupoid as described in \cite{Hawkins}. The geometric quantization of this Lagrangian produces an $(\infty, 1)$-category with a distinguished object. The algebra of endomorphisms of this object should be considered as an approximation to deformation quantization of $X$ \cite{WeinsteinXu}.

\item Given a Hamiltonian $G$-scheme $X$, the moment map $\mu\colon X\rightarrow \g^*$ induces a $1$-shifted Lagrangian morphism $[X/G]\rightarrow [\g^*/G]$. We show (\cref{thm:prequantumHamiltonianspace}) that a $G$-equivariant prequantization and polarization of $X$ gives rise to a prequantization and polarization of this Lagrangian. This gives a natural perspective on the prequantization of the (derived) Hamiltonian reduction.

\item Let $G$ be a split semisimple group and $x\in\g^*$ a semisimple element. Its stabilizer is a Levi subgroup $L\subset P$ of a parabolic subgroup $P$. In fact, the coadjoint orbit $\cO$ of $x$ is equivalent to a twisted cotangent bundle of the partial flag variety $G/P$ \cite{Lisiecki}. By the above, the morphism $[\cO/G]\rightarrow [\g^*/G]$ has a $1$-shifted Lagrangian structure and we describe explicitly the corresponding polarization and prequantization in \cref{prop:semisimplecoadjoint}.

\item Given a nonzero nilpotent orbit $\cO\subset \g^*$, there is a slice $\cS\subset \g^*$ for the coadjoint action passing through the orbit $\cO$ known as the Slodowy slice. We construct a natural polarization on the corresponding $1$-shifted Lagrangian morphism $\cS\rightarrow [\g^*/G]$ in \cref{prop:slodowy}.

\item If $X$ is an $n$-shifted symplectic stack and $C$ a smooth projective curve, we prove (see \cref{prop:deRhamMap}) that the moduli stack of locally-constant maps $\Map(C_{\dR}, X)$ is a twisted cotangent bundle of the moduli stack of algebraic maps $\Map(C, X)$. In this way we show (see \cref{thm:ChernSimons}) that the geometric quantization of the moduli stack $\LocSys_G(C)$ of flat connections on $C$ is given by sections of a specific line bundle over the moduli stack $\Bun_G(C)$ of $G$-bundles.
\end{itemize}

\subsection*{Further examples}

There are some examples where the output of the geometric quantization is clear, but a proper treatment of these examples requires one to go beyond the setting of shifted symplectic structures in \cite{PTVV}.

\begin{enumerate}
\item Often supersymmetric field theories are only $\Z/2$-graded rather than $\Z$-graded. So, one is often faced with phase spaces which are $\Z/2$-graded supermanifolds. One may model these algebro-geometrically by working over the ground ring $k[u, u^{-1}]$, where $\deg(u) = 2$. In this case an $n$-shifted symplectic structure $\omega$ may be turned into an $(n+2)$-shifted symplectic structure $u\omega$. Note that this forces one to work with non-connective commutative dg algebras, i.e. we have to work in the setting of $D$-stacks.

For instance, consider a smooth scheme $X$ with a function $f\colon X\rightarrow \bA^1$. The derived critical locus $\Crit(f)$ of $f$ viewed as a $1$-shifted symplectic space is the phase space of the Landau--Ginzburg B-model on the point. Recall that $\Crit(f)$ is a twisted $(-1)$-shifted cotangent bundle of $X$ (see \cref{ex:dCrit}). We expect that it admits a prequantization as a $1$-shifted symplectic stack given by a gerbe that we will formally write as $\exp(uf)$. What should $\IndCoh^{\exp(uf)}(X)$ be?

Let us recall the picture for the $\infty$-category of matrix factorizations $\MF(X, f)$ given by Preygel \cite{Preygel}. The zero fiber $f^{-1}(0)$ carries an action of the derived group scheme $\Omega\bA^1$. Recall the $\infty$-category of ind-coherent sheaves $\IndCoh$ introduced in \cite{GaitsgoryIndCoh}. Then $\IndCoh(f^{-1}(0))$ carries an action of $\IndCoh(\Omega\bA^1)\cong \Mod_{k[\beta]}$, where $\deg(\beta) = 2$. One then has
\[\Ind\MF(X, f) = \IndCoh(f^{-1}(0))\otimes_{\Mod_{k[\beta]}} \Mod_{k[\beta, \beta^{-1}]}.\]
Let $\widehat{X}_{f^{-1}(0)}$ be the formal completion of $X$ along $f^{-1}(0)$. We expect that $\IndCoh^{\exp(uf)}(\widehat{X}_{f^{-1}(0)})$ coincides with the above $\infty$-category, where the twist by the gerbe $\exp(uf)$ identifies the variable $u$ in the ground ring with the variable $\beta$ in $\IndCoh(\Omega\bA^1)$. This provides an explanation why the category of boundary conditions in the Landau--Ginzburg B-model is given by the category of matrix factorizations \cite{KapustinLi}.

\item The definition of shifted symplectic structures in \cite{PTVV} is given for a derived Artin stack locally of finite presentation, which, in particular, implies that the stack admits a perfect cotangent complex. For a derived prestack $X$ we may ask, more generally, for $X$ to admit a pro-cotangent complex $\bL_X^{pro}\in\mathrm{ProQCoh}(X)$ (where $\mathrm{ProQCoh}(X)$ is the right Kan extension of the functor $A\mapsto \mathrm{Pro}(\Mod_A)$ defined on derived affine schemes), see \cite[Chapter 1]{GaitsgoryRozenblyum2}. Although pro-quasicoherent sheaves do not have a reasonable notion of self-duality, there is a subcategory $\mathrm{Tate}(X)\subset \mathrm{ProQCoh}(X)$ of Tate complexes which does admit self-duality. So, we may consider \emph{Tate prestacks} which are derived prestacks which admit a Tate cotangent complex $\bL_X^{Tate}$. We refer to \cite{Hennion,Heleodoro} for a related formalism. We expect that one may extend the definition of $n$-shifted symplectic structures to Tate prestacks; this is a joint work in progress with V. Melani and M. Porta.

Consider the stack $\LocSys_G(\mathring{D})$ of $G$-bundles with a connection on the ``formal'' punctured disk $\mathring{D} = \Spec k(\!(t)\!)$ (see \cite{Raskin} for a thorough study of this stack). We expect that \cref{prop:deRhamMap} extends to Tate prestacks, so that one may identify $\LocSys(\mathring{D})$ with a twisted 1-shifted cotangent bundle of $\B LG$, where $LG$, the loop group, is the group of maps $\mathring{D}\rightarrow G$. In particular, the geometric quantization of of $\LocSys(\mathring{D})$ should be the category of representations of the loop group $LG$ at a given level. $\LocSys_G(\mathring{D})$ is an algebro-geometric avatar of the phase of the classical Chern--Simons theory on the circle; in particular, this gives an explanation why the category of line operators in the quantum Chern--Simons theory is the category of $LG$-representations. We refer to \cite{FreedTeleman} for an interesting related statement.
\end{enumerate}

\subsection*{Acknowledgements}

I would like to thank Justin Hilburn and Theo Johnson-Freyd for useful conversations. This research was partially supported by the NCCR SwissMAP grant of the Swiss National Science Foundation.

\section{Polarizations}

Let $\cS$ be the $\infty$-category of small spaces (equivalently, $\infty$-groupoids). Throughout the paper we fix $k$, a field of characteristic zero. Let $\CAlg$ be the $\infty$-category of commutative dg $k$-algebras and $\CAlg^{\leq 0}\subset \CAlg$ the full subcategory of those concentrated in non-positive cohomological degrees. Recall that a derived prestack is a functor $\CAlg^{\leq 0}\rightarrow \cS$. A derived stack is a derived prestack satisfying \'{e}tale descent. \cite{HAGII,GaitsgoryRozenblyum1,SAG} are the standard sources on derived algebraic geometry.

\subsection{Differential forms on stacks}

Let $\pi\colon X\rightarrow S$ be a morphism of derived Artin stacks. Recall from \cite{PTVV, CPTVV} the notion of the relative de Rham algebra $\DR(X/S)$; it is a graded mixed commutative dg algebra which, as a plain graded commutative dg algebra, is
\[\DR(X/S)\cong \Gamma(X, \Sym(\bL_{X/S}[-1])),\]
where $\bL_{X/S}$ is the relative cotangent complex. Besides the cohomological grading, we also have the \emph{weight} grading by the form degree and besides the cohomological differential $\d$, we also have the de Rham differential $\ddr$ increasing both the cohomological and form degrees.

\begin{defn}
Let $X\rightarrow S$ be a morphism of derived Artin stacks.
\begin{itemize}
\item A \defterm{relative $p$-form of degree $n$ on $X\rightarrow S$} is a $\d$-closed element $\omega_p$ of $\DR(X/S)$ of weight $p$ and cohomological degree $p+n$. We denote by $\cA^p(X/S, n)$ the space of such.

\item A \defterm{closed relative $p$-form of degree $n$ on $X\rightarrow S$} is a $(\d+\ddr)$-closed element $\omega_p+\omega_{p+1}+\dots$ of $\DR(X/S)$, where $\omega_q$ is a $q$-form of degree $n-(q-p)$. We denote by $\cA^{p, \cl}(X/S, n)$ the space of such.
\end{itemize}
\end{defn}

If we assume $X\rightarrow S$ is locally of finite presentation, the relative cotangent complex $\bL_{X/S}$ is perfect, so we can define the tangent complex $\bT_{X/S} = (\bL_{X/S})^\vee$ as the dual. A relative two-form of degree $n$ induces a morphism
\[\omega^\sharp\colon \bT_{X/S}\longrightarrow \bL_{X/S}[n]\]
from the tangent to the cotangent complex.

\begin{defn}
A \defterm{relative $n$-shifted symplectic structure on $X\rightarrow S$} is a closed relative two-form $\omega$ of degree $n$, such that $\omega^\sharp\colon \bT_{X/S}\rightarrow \bL_{X/S}[n]$ is a quasi-isomorphism.
\end{defn}

The notion of shifted isotropic and shifted Lagrangian structures from \cite{PTVV} extends verbatim to the relative setting, so we will not repeat the definitions. Moreover, the proofs in \cite{PTVV,AmorimBenBassat,BenBassat,CHS} that $n$-shifted Lagrangian intersections carry an $(n-1)$-shifted symplectic structure go through without modifications.

\subsection{Lagrangian fibrations}

Let us recall the following notion introduced in \cite{CalaqueCotangent}.

\begin{defn}
Let $\pi\colon X\rightarrow B$ be a map of derived Artin stacks. The structure of an \defterm{$n$-shifted isotropic fibration on $\pi$} is an $n$-shifted presymplectic structure on $X$ together with a trivialization of its image under $\cA^{2, cl}(X, n)\rightarrow \cA^{2, cl}(X/B, n)$.
\end{defn}

Now assume $X$ is locally of finite presentation and consider the map $\bT_X\cong \bL_X[n]\rightarrow \bL_{X/B}[n]$. Given an $n$-shifted isotropic fibration $\pi\colon X\rightarrow B$, we have a natural sequence
\begin{equation}
\bT_{X/B}\longrightarrow \bT_X\longrightarrow \bL_{X/B}[n]
\label{eq:ndisotropicfibration}
\end{equation}
in $\QCoh(X)$.

\begin{defn}
An \defterm{$n$-shifted Lagrangian fibration} $\pi\colon X\rightarrow B$ is an $n$-shifted isotropic fibration such that \eqref{eq:ndisotropicfibration} is a fiber sequence.
\end{defn}

See \cite[Lemma 1.3]{CalaqueCotangent} for the following statement.
\begin{lm}
Suppose $\pi\colon X\rightarrow B$ is an $n$-shifted Lagrangian fibration. Then $X$ is $n$-shifted symplectic.
\end{lm}

We can reformulate the data of an $n$-shifted isotropic fibration in the following way. Observe that the morphism $\DR(X)\rightarrow \DR(X/B)$ factors as
\[\DR(X)\longrightarrow \DR(X\times B/B)\longrightarrow \DR(X/B),\]
where the first morphism regards a differential form on $X$ as a relative differential form on $X\times B\rightarrow B$ constant along $B$ and the second morphism is the pullback along $(\id\times \pi)\colon X\rightarrow X\times B$.

\begin{prop} $ $
\begin{enumerate}
\item An $n$-shifted isotropic fibration $f\colon X\rightarrow B$ is the same as a pair of an $n$-shifted presymplectic structure on $X$ and an $n$-shifted isotropic structure on $(\id\times \pi)\colon X\rightarrow X\times B$ relative to $B$.
\item An $n$-shifted isotropic fibration $\pi$ is Lagrangian if, and only if, $(\id\times \pi)\colon X\rightarrow X\times B$ is Lagrangian.
\end{enumerate}
\label{prop:Lagrangianfibrationequivalent}
\end{prop}
\begin{proof} $ $
\begin{enumerate}
\item The null-homotopy of the image of $\omega$ under $\DR(X)\rightarrow \DR(X/B)$ is the same as the null-homotopy of the image of $\omega$ under $\DR(X\times B/B)\rightarrow \DR(X/B)$.

\item The map $(\id\times \pi)\colon X\rightarrow X\times B$ is Lagrangian if, and only if,
\[\bT_{X/B}\longrightarrow \bT_{X\times B/B}\longrightarrow \bL_{X/B}[n]\]
is a fiber sequence. But $\bT_{X\times B/B}\cong \bT_X$. So it is a fiber sequence if, and only if, \eqref{eq:ndisotropicfibration} is a fiber sequence.
\end{enumerate}
\end{proof}

\begin{cor}
Suppose $\pi\colon X\rightarrow B$ is an $n$-shifted Lagrangian fibration. For any $b\in B$ the inclusion of the fiber $X_b\rightarrow X$ carries an $n$-shifted Lagrangian structure.
\end{cor}
\begin{proof}
By \cref{prop:Lagrangianfibrationequivalent} we have an $n$-shifted Lagrangian structure on $X\rightarrow X\times B$ relative to $B$. Considering the base change along the map $b\colon \pt\rightarrow B$ we obtain an $n$-shifted Lagrangian structure on $\pt\times_B X\rightarrow X$.
\end{proof}

\begin{example}
Suppose $X$ is a derived Artin stack locally of finite presentation. It is shown in \cite{CalaqueCotangent} that the $n$-shifted cotangent stack $\T^*[n] X$ is $n$-shifted symplectic. Moreover, the natural projection $\T^*[n] X\rightarrow X$ is an $n$-shifted Lagrangian fibration.
\label{ex:shiftedcotangent}
\end{example}

Let us mention that \cref{prop:Lagrangianfibrationequivalent} suggests a slight generalization of the notion of a Lagrangian fibration, see \cite{GuilleminSternbergGeneralization}.

\begin{defn}
An \defterm{$n$-shifted Lagrangian triple} is a correspondence of derived Artin stacks $X\leftarrow Z\rightarrow B$ together with an $n$-shifted symplectic structure on $X$ and an $n$-shifted Lagrangian structure on $Z\rightarrow X\times B$ relative to $B$.
\label{def:Lagrangiantriple}
\end{defn}

There are the following examples:
\begin{itemize}
\item For $Z=X$ we recover the notion of an $n$-shifted Lagrangian fibration.

\item For $B = \pt$ we recover the notion of an $n$-shifted Lagrangian.
\end{itemize}

\subsection{Intersections of Lagrangian fibrations}

Intersections of Lagrangian fibrations behave in a way analogous to intersections of Lagrangians. Let us begin with the following observation (see also \cite[Proposition 3.3]{Grataloup}).

\begin{prop}
Suppose $\pi\colon X\rightarrow B$ is an $n$-shifted Lagrangian fibration and $f\colon L\rightarrow X$ is an $n$-shifted Lagrangian. Then $L$ carries an $(n-1)$-shifted symplectic structure relative to $B$ via the projection $L\rightarrow X\rightarrow B$.
\label{prop:fibrationintersection}
\end{prop}
\begin{proof}
We have a pullback diagram
\[
\xymatrix{
L \ar^{f}[d] \ar^-{\id\times \pi\circ f}[r] & L\times B \ar^{f\times \id}[d] \\
X \ar^-{\id\times \pi}[r] & X\times B.
}
\]

Since both $X\rightarrow X\times B$ and $L\times B\rightarrow X\times B$ are Lagrangian relative to $B$, by \cite[Section 2.2]{PTVV} we obtain that $L$ carries an $(n-1)$-shifted symplectic structure relative to $B$ since it is a Lagrangian intersection.
\end{proof}

This motivates the following definition.

\begin{defn}
Suppose $X$ is an $n$-shifted symplectic stack and $f\colon L\rightarrow X$ is an $n$-shifted Lagrangian morphism. An \defterm{$n$-shifted Lagrangian fibration} on $L\rightarrow X$ is the following triple:
\begin{enumerate}
\item An $n$-shifted Lagrangian fibration $X\rightarrow B_X$ inducing an $(n-1)$-shifted symplectic structure on $L$ relative to $B_X$.

\item A commutative diagram
\begin{equation}
\xymatrix{
L \ar[d] \ar[r] & X \ar[d] \\
B_L \ar[r] & B_X
}
\label{eq:Lagrangianfibration}
\end{equation}
of derived stacks.

\item An $(n-1)$-shifted Lagrangian fibration structure on $L\rightarrow B_L$ relative to $B_X$.
\end{enumerate}
\end{defn}

\begin{example}
Suppose $X\rightarrow B$ is an $n$-shifted Lagrangian fibration and $L\rightarrow X$ is an $n$-shifted Lagrangian such that the composite $L\rightarrow X\rightarrow B$ is \'{e}tale. In this case we say the Lagrangian $L$ is \defterm{transverse} to the Lagrangian fibration $X\rightarrow B$. Then the diagram
\begin{equation}
\xymatrix{
L \ar[r] \ar[d] & X \ar[d] \\
B \ar@{=}[r] & B
}
\label{eq:transverseLagrangian}
\end{equation}
realizes an $n$-shifted Lagrangian fibration on $L\rightarrow X$. In other words, if the Lagrangian $L$ is transverse to the Lagrangian fibration $X\rightarrow B$, there is a canonical choice of a Lagrangian fibration on $L$.
\label{ex:transverseLagrangian}
\end{example}

\begin{example}
Suppose $Y\rightarrow X$ is a morphism of derived Artin stacks locally of finite presentation. It is shown in \cite{CalaqueCotangent} that
\[
\xymatrix{
\N^*[n]Y \ar[d] \ar[r] & \T^*[n] X \ar[d] \\
Y \ar[r] & X
}
\]
is an $n$-shifted Lagrangian fibration on $\N^*[n]Y\rightarrow \T^*[n] X$.
\label{ex:shiftedconormal}
\end{example}

Let us finally study intersections of Lagrangian fibrations.

\begin{thm}
Let $X\rightarrow B_X$ be an $n$-shifted Lagrangian fibration, $L_1, L_2\rightarrow X$ are two $n$-shifted Lagrangians equipped with Lagrangian fibrations $L_i\rightarrow B_{L_i}$ resulting in the commutative diagram
\[
\xymatrix@R=0cm{
L_1 \ar[dr] \ar[dd] && L_2 \ar[dl] \ar[dd] \\
& X \ar[dd] & \\
B_{L_1} \ar[dr] && B_{L_2} \ar[dl] \\
& B_X &
}
\]
Then $L_1\times_X L_2$ carries an $(n-1)$-shifted symplectic structure together with a Lagrangian fibration
\[L_1\times_X L_2\rightarrow B_{L_1}\times_{B_X} B_{L_2}.\]
\label{thm:Lagrangianfibrationintersection}
\end{thm}
\begin{proof}
The $n$-shifted Lagrangian structures on $L_i\rightarrow X$ and the $n$-shifted Lagrangian fibration on $X\rightarrow B_X$ give three $n$-shifted Lagrangian morphisms
\[
\xymatrix{
L_1\times B_X \ar[dr] & X \ar[d] & L_2\times B_X \ar[dl] \\
& X\times B_X
}
\]
relative to $B_X$. Therefore, using \cite[Theorem 3.1]{BenBassat} one obtains an $(n-1)$-shifted Lagrangian structure on the morphism
\[
L_1\times_X L_2\longrightarrow L_1\times_{B_X} L_2\times L_1\times_X L_2
\]
relative to $B_X$. The $(n-1)$-shifted symplectic structure on the right is given by the product of the natural $(n-1)$-shifted symplectic structure on the intersection $L_1\times_X L_2$ and the relative $(n-1)$-shifted symplectic structures on $L_i\rightarrow B_X$ given by \cref{prop:fibrationintersection}. So, it can be presented as an $(n-1)$-shifted Lagrangian correspondence
\begin{equation}
\xymatrix{
& L_1\times_X L_2 \ar[dl] \ar[dr] & \\
L_1\times_X L_2\times B_X && L_1\times_{B_X} L_2
}
\label{eq:Lagrangianfibrationcorrespondence}
\end{equation}
relative to $B_X$.

Denote
\[B = B_{L_1}\times_{B_X} B_{L_2}.\]
The Lagrangian fibration structure on $L_1$ gives an $(n-1)$-shifted Lagrangian structure on the morphism
\[L_1\longrightarrow L_1\times_{B_X} B_{L_1}\]
relative to $B_{L_1}$. Its base change along $B\rightarrow B_{L_1}$ gives an $(n-1)$-shifted Lagrangian structure on the morphism
\[L_1\times_{B_X} B_{L_2}\longrightarrow L_1\times_{B_X} B\]
relative to $B$. Taking the relative fiber product over $B$ with a similar Lagrangian for $L_2$ we obtain an $(n-1)$-shifted Lagrangian structure on the morphism
\[L_1\times_{B_X} L_2\longrightarrow L_1\times_{B_X} L_2\times_{B_X} B\]
relative to $B$. Combining it together with \eqref{eq:Lagrangianfibrationcorrespondence} we obtain a pair of $(n-1)$-shifted Lagrangian correspondences
\[
\xymatrix{
& L_1\times_X L_2\times_{B_X} B \ar[dl] \ar[dr] && L_1\times_{B_X} L_2 \ar[dl] \ar[dr] & \\
L_1\times_X L_2\times B && L_1\times_{B_X} L_2\times_{B_X} B && B
}
\]
relative to $B$. The composition of these two Lagrangian correspondences finally gives an $(n-1)$-shifted Lagrangian structure on the morphism
\[
L_1\times_X L_2\longrightarrow L_1\times_X L_2\times B,
\]
i.e. a Lagrangian fibration structure on $L_1\times_X L_2\rightarrow B$.
\end{proof}

\begin{remark}
In the case the Lagrangians $L_i\rightarrow X$ are transverse to $X\rightarrow B_X$ and we use the natural Lagrangian fibrations on $L_i$ from \cref{ex:transverseLagrangian} we recover \cite[Theorem 3.4]{Grataloup}.
\end{remark}

\subsection{Twisted cotangent bundles}

In this section we slightly generalize \cref{ex:shiftedcotangent,ex:shiftedconormal} to the case of \emph{twisted} shifted cotangent and shifted conormal bundles.

Let $X$ be a derived Artin stack locally of finite presentation and $\alpha$ a closed one-form on $X$ of degree $(n+1)$. Then one has a natural $(n+1)$-shifted isotropic structure on the graph of $\alpha$
\[\Gamma_\alpha\colon X\longrightarrow \T^*[n+1] X.\]
Namely, if we denote by $\lambda$ the Liouville one-form on $\T^*[n+1] X$, then $\Gamma_\alpha^* \lambda = \alpha$ and hence $\Gamma_\alpha^* \ddr \lambda = \ddr\alpha = 0$. Analogously to \cite[Corollary 2.4]{CalaqueCotangent} one can show that this isotropic structure is, in fact, Lagrangian. Thus, we obtain two $(n+1)$-shifted Lagrangian morphisms
\[
\Gamma_0, \Gamma_\alpha\colon X\longrightarrow \T^*[n+1] X,
\]
so their intersection has an $n$-shifted symplectic structure.

\begin{defn}
Let $X$ be a derived Artin stack locally of finite presentation equipped with a closed one-form $\alpha$ of degree $(n+1)$. The \defterm{$n$-shifted $\alpha$-twisted cotangent bundle of $X$} is the fiber product
\[
\xymatrix{
\T_\alpha^*[n] X \ar[r] \ar[d] & X \ar^{\Gamma_0}[d] \\
X \ar^-{\Gamma_\alpha}[r] & \T^*[n+1] X
}
\]
\end{defn}

\begin{remark}
If $\alpha$ is not closed, one can still make sense of the twisted cotangent bundle, but it will no longer carry a shifted symplectic structure.
\end{remark}

\begin{example}
Let $X$ be a smooth scheme and $B$ a closed two-form on $X$ which we may consider as a closed one-form on $X$ of degree $1$. Then $\T_B^* X$ is isomorphic to $\T^* X$ as a scheme with the symplectic structure given by $\omega_{T^* X} + \pi^* B$, where $\pi\colon \T^* X\rightarrow X$ is the natural projection. This is known as a \emph{magnetic deformation} of the cotangent bundle.
\end{example}

\begin{example}
Let $X$ be a smooth scheme and $\alpha\in\H^1(X, \Omega^{\geq 1})$ an arbitrary closed one-form of degree 1. Choose an affine open cover $\{U_i\}$ of $X$. Then $\alpha$ is represented by a collection of closed two-forms $B_i$ on $U_i$ together with a collection of one-forms $\alpha_{ij}$ on the overlaps $U_i\cap U_j$ satisfying $\ddr \alpha_{ij} = B_i - B_j$ and a cocycle condition on the triple intersections. Then $\T_\alpha^* X$ is given by gluing the magnetic cotangent bundles $\T^*_{B_i} U_i$ using $\alpha_{ij}$. We refer to \cite[Section A.1]{BeilinsonKazhdan} for more details on twisted cotangent bundles in the underived setting.
\end{example}

\begin{example}
Suppose $X$ is a smooth scheme and $S\colon X\rightarrow \bA^1$ a function. Then the one-form $\ddr S$ is closed, so we may consider the $(-1)$-shifted twisted cotangent bundle $\T_{\ddr S}^*[-1]X$. This is known as the \emph{derived critical locus} of $S$, see \cite{VezzosidCrit} and \cite[Chapter 4.1]{CostelloGwilliam1}.
\label{ex:dCrit}
\end{example}

\begin{prop}
Let $X$ be a derived Artin stack locally of finite presentation and $\alpha$ a closed one-form of degree $(n+1)$. The projection $\T^*_\alpha[n] X\rightarrow X$ has a natural structure of a Lagrangian fibration.
\label{prop:twistedcotangentfibration}
\end{prop}
\begin{proof}
The projection $\T^*[n+1] X\rightarrow X$ has a natural structure of a Lagrangian fibration. The graphs $\Gamma_\alpha, \Gamma_0\colon X\rightarrow \T^*[n+1] X$ are transverse to the Lagrangian fibration. In particular, they carry an obvious Lagrangian fibration. Therefore, by \cref{thm:Lagrangianfibrationintersection} their intersection carries a natural Lagrangian fibration.
\end{proof}

Let us now consider a relative situation. Let $f\colon Y\rightarrow X$ be a morphism of derived stacks, $\alpha$ a closed one-form of degree $(n+1)$ on $X$ and $h\colon f^*\alpha\sim 0$ is a nullhomotopy. We have three Lagrangians $\Gamma_0\colon X\rightarrow \T^*[n+1]X$, $\Gamma_\alpha\colon X\rightarrow \T^*[n+1] X$ and $\N^*[n+1] Y\rightarrow \T^*[n+1] X$. We may identify \[\Gamma_\alpha\times_{\T^*[n+1] X} \N^*[n+1] Y\cong \T^*_{f^*\alpha} Y.\]
From \cite[Theorem 3.1]{BenBassat} we conclude that
\[
\xymatrix{
& \Gamma_0\times_{\T^*[n+1] X} \N^*[n+1] Y\times_{\T^*[n+1] X} \Gamma_\alpha \ar[dl] \ar[dr] & \\
\T^*_\alpha[n] X && \T^*[n] Y\times \T^*_{f^*\alpha}[n] Y
}
\]
is an $n$-shifted Lagrangian correspondence. Using the homotopy $h$, the diagonal map of $\T^*[n] Y$ carries a Lagrangian structure. Intersecting it with the above morphism we obtain an $n$-shifted Lagrangian morphism
\[\N^*_\alpha[n] Y\longrightarrow \T^*_\alpha[n] X.\]

\begin{defn}
Let $f\colon Y\rightarrow X$ be a morphism of derived Artin stacks locally of finite presentation, $\alpha$ a closed one-form of degree $(n+1)$ on $X$ and $f^*\alpha\sim 0$ a nullhomotopy on $Y$. The \defterm{$n$-shifted $\alpha$-twisted conormal bundle} of $f$ is the morphism
\[\N^*_\alpha[n]Y \longrightarrow \T^*_\alpha[n] X.\]
\end{defn}

Analogously to \cref{prop:twistedcotangentfibration}, but with more work, we expect one can prove that the diagram
\[
\xymatrix{
\N^*_\alpha[n] Y \ar[r] \ar[d] & \T^*_\alpha[n] X \ar[d] \\
Y \ar[r] & X
}
\]
defines a Lagrangian fibration on the $n$-shifted twisted conormal bundle.

\subsection{Spaces of locally-constant maps}
\label{sect:deRham}

Let $C$ be a smooth scheme. Gaitsgory and Rozenblyum \cite[Chapter 8]{GaitsgoryRozenblyum2} define an $\infty$-category $\LieAlgbroid(C)$ of Lie algebroids on $C$, so that any Lie algebroid $\cL$ on $C$ defines a formal moduli problem $C\rightarrow [C/\cL]$. Moreover, there is a natural free-forgetful adjunction
\[
\xymatrix{
\free\colon \QCoh(C)_{/\T_C} \ar@<.5ex>[r] & \LieAlgbroid(C)\colon \oblv \ar@<.5ex>[l]
}
\]
where $\QCoh(C)/_{/\T_C}$ is the $\infty$-category of quasi-coherent sheaves $\cE$ equipped with an anchor map $\cE\rightarrow \T_C$.

\begin{lm}
Suppose $\cL$ is a Lie algebroid on $C$ which is a line bundle concentrated in degree $0$. Then the natural morphism
\[\free(\oblv(\cL))\longrightarrow \cL\]
is an isomorphism.
\label{lm:free1dimLieAlgebroid}
\end{lm}
\begin{proof}
The forgetful functor $\oblv$ is conservative, so it is enough to show that
\[\oblv(\free(\oblv(\cL)))\longrightarrow \oblv(\cL)\]
is an isomorphism. By the adjunction axioms it is equivalent to proving that the unit of the monad $\top=\oblv\circ\free$ induces an isomorphism
\[\oblv(\cL)\longrightarrow \top(\oblv(\cL)).\]
For this let us recall from \cite[Proposition 5.3.2]{GaitsgoryRozenblyum1} that the monad $\top$ admits a nonnegative filtration whose associated graded is $\oblv_{\Lie}\circ\free_{\Lie}$, where $\free_{\Lie}\dashv \oblv_{\Lie}$ is the free-forgetful adjunction for Lie algebras. But since $\cL$ is a line bundle, $\cL\rightarrow \oblv_{\Lie}(\free_{\Lie}(\cL))$ is an isomorphism.
\end{proof}

Recall the notion of the de Rham stack (see e.g. \cite[Definition 2.1.3]{CPTVV}) $C_{\dR}$, which can be identified as
\[C_{\dR} \cong [C/\T_C].\]

Suppose $C$ is a smooth and proper curve. As in \cite[Section 2.1]{PTVV}, for any derived stack $Y$ one can construct a graded mixed morphism
\[\DR(Y\times C)\longrightarrow \DR(Y)\otimes \DR(C).\]
The Serre duality provides an integration map
\[\int_C\colon \DR(C)\longrightarrow k(1)[-1],\]
which is a map of graded mixed complexes. If $(X, \omega)$ is an $n$-shifted symplectic stack, we denote by $\int_C\ev^*\omega$ the closed one-form of degree $(n-1)$ on $\Map(C, X)$ obtained as the image of $\omega$ under
\[\DR(X)\xrightarrow{\ev^*}\DR(\Map(C, X)\times C)\longrightarrow \DR(\Map(C, X))\otimes \DR(C)\longrightarrow \DR(\Map(C, X))\otimes k(1)[-1].\]

\begin{prop}
Let $C$ be a smooth and proper curve and $(X, \omega)$ an $n$-shifted symplectic stack. One has an isomorphism of derived stacks
\[\Map(C_{\dR}, X)\cong \T^*_{\int_C \ev^*\omega}[n-2]\Map(C, X)\]
compatible with the natural projections to $\Map(C, X)$.
\label{prop:deRhamMap}
\end{prop}
\begin{proof}
For any map of derived stacks $f\colon C\rightarrow X$ the differential defines an element
\[\ddr f\in\Gamma(C, f^*\bT_X\otimes \T^*_C).\]
There is an equivalence
\[\bT_f\Map(C, X)\cong\Gamma(C, f^*\bT_X),\]
so by Serre duality
\[\bL_f\Map(C, X)\cong \Gamma(C, f^*\bL_X\otimes \T^*_C)[1].\]
In particular, the composition $\omega\circ\ddr f$ defines a morphism of derived stacks
\[\omega\circ\ddr f\colon \Map(C, X)\longrightarrow \T^*[n-1]\Map(C, X).\]
By \cref{lm:free1dimLieAlgebroid} we have an isomorphism
\[C_{\dR}\cong \RealSqZ(\T_C\xrightarrow{\id}\T_C),\]
where $\RealSqZ$ is the formal moduli problem under $C$ introduced in \cite[Chapter 8, Section 5.1]{GaitsgoryRozenblyum2}. Therefore, by \cite[Chapter 8, Lemma 5.1.3]{GaitsgoryRozenblyum2} the derived stack $\Map(C_{\dR}, X)$ is isomorphic to the intersection of $\omega\circ \ddr f$ and the zero section. It is then easy to see that $\omega\circ \ddr f$ coincides with the underlying one-form of $\int_C \ev^*\omega$.
\end{proof}

Note that by \cite[Theorem 2.5]{PTVV} the mapping stack $\Map(C_{\dR}, X)$ has an $(n-2)$-shifted symplectic structure.

\begin{conjecture}
The isomorphism
\[\Map(C_{\dR}, X)\cong \T^*_{\int_C\ev^*\omega}[n-2] \Map(C, X)\]
constructed in \cref{prop:deRhamMap} is compatible with the $(n-2)$-shifted symplectic structures on both sides.
\end{conjecture}

Note that it is straightforward to check that the underlying two-forms on both sides coincide and the difficult part is in comparing the closure data.

\begin{example}
Let $G$ be a reductive algebraic group equipped with a nondegenerate invariant symmetric bilinear pairing $\langle-, -\rangle$ on its Lie algebra. Then $\B G$ carries a canonical 2-shifted symplectic structure $\omega$. The space
\[\LocSys_G(C) = \Map(C_{\dR}, \B G)\]
is the moduli stack of principal $G$-bundles with a flat connection on $C$. If $\langle-,-\rangle$ is the Killing form, by the Grothendieck--Riemann--Roch theorem it is easy to see that the closed one-form of degree $1$ $\int_C\ev^*\omega$ on the moduli stack $\Bun_G(C) = \Map(C, \B G)$ coincides with the first Chern class of the determinant line bundle on $\Bun_G(C)$. \Cref{prop:deRhamMap} identifies
\[\LocSys_G(C)\cong \T^*_{\int_C\ev^*\omega} \Bun_G(C).\]
See \cite[Proposition 4.1.4]{BenZviFrenkel} for a related statement.
\end{example}

\section{Prequantizations}

\subsection{Gerbes with connections}

In this section we introduce gerbes with connective structures on derived stacks. We refer to \cite{Brylinski,Gajer,BreenMessing} for the classical theory. The reference \cite{FRS} deals with higher gerbes in the smooth setting and \cite{Wallbridge} promotes essentially the same definition to the setting of derived stacks.

Let $\cA^p(n)$ and $\cA^{p, \cl}(n)$ be the derived stacks of $p$-forms of degree $n$ and closed $p$-forms of degree $n$ introduced in \cite{PTVV}. By construction we have equivalences
\[\cA^{1, \cl}(n)\cong \Omega\cA^{1, \cl}(n+1),\qquad \cA^1(n)\cong \Omega \cA^1(n+1)\]
for any $n\in\Z$. By adjunction we get maps $\B\cA^1(n)\rightarrow \cA^1(n+1)$.

\begin{lm}
Suppose $n\geq 0$. Then the map $\B\cA^1(n)\rightarrow \cA^1(n+1)$ is an equivalence.
\end{lm}
\begin{proof}
Suppose $S$ is an affine derived scheme. Any one-form on $S$ of positive cohomological degree is zero, so $\Map(S, \B\cA^1(n))\rightarrow \Map(S, \cA^1(n+1))$ is essentially surjective. It is fully faithful since $\cA^1(n)\rightarrow \Omega \cA^1(n+1)$ is an equivalence.
\end{proof}

Iterating the previous lemma, we obtain that $\B^n\cA^1(0)\rightarrow\cA^1(n)$ is an equivalence.

\begin{remark}
The map $\B^n\cA^{1, \cl}(0)\rightarrow \cA^{1, \cl}(n)$ is not an equivalence for $n>0$. Indeed, the inclusion $\cA^{n+1, \cl}(0)\rightarrow \cA^{1, \cl}(n)$ is an equivalence. However, the closed two-form $(xy)^{-1}\ddr x\wedge \ddr y$ on $\Gm\times \Gm$ is not \'{e}tale-locally exact.
\end{remark}

We have a morphism of abelian group stacks
\[\dlog\colon \Gm\longrightarrow \cA^{1, \cl}(0)\]
which on a commutative dg algebra $R$ sends an invertible element $f\in R^\times$ to $(\ddr f) / f$. 

Delooping, we get a morphism of abelian group stacks
\[c_1\colon \B^n \Gm\longrightarrow \B^n \cA^{1, \cl}(0)\longrightarrow \cA^{1, \cl}(n).\]
Post-composing the above map with the forgetful map $\cA^{1, \cl}(n)\rightarrow \cA^1(n)$ we get a map $\B^n\Gm\rightarrow \cA^1(n)$.

\begin{defn}
Let $n\geq 0$ and $X$ a derived stack.
\begin{itemize}
\item An \defterm{$n$-gerbe} $\cG$ on $X$ is a map $X\rightarrow\B^{n+1}\Gm$.
\item An \defterm{$n$-gerbe with a connective structure} $(\cG, \nabla)$ on $X$ is an $n$-gerbe $\cG$ on $X$ together with a nullhomotopy $\nabla$ of $c_1(\cG)\in\cA^1(X, n+1)$.
\item An \defterm{$n$-gerbe with a flat connection} $(\cG, \nabla)$ on $X$ is an $n$-gerbe $\cG$ on $X$ together with a nullhomotopy $\nabla$ of $c_1(\cG)\in\cA^{1, \cl}(X, n+1)$.
\end{itemize} 
\end{defn}

Consider the diagram
\[
\xymatrix{
& \cA^{2, \cl}(n) \ar[r] \ar[d] & 0 \ar[d] \\
\B^{n+1}\Gm \ar^{c_1}[r] & \cA^{1, \cl}(n+1) \ar[r] & \cA^1(n+1)
}
\]
of abelian group stacks. The bottom-right square is a pullback, so any gerbe $(\cG, \nabla)$ with a connective structure on $X$ has a \defterm{curvature} $\curv(\cG, \nabla)\in\cA^{2, cl}(X, n)$.

\begin{example}
The stack $\B\Gm$ classifies line bundles and the morphism $c_1\colon \B\Gm\rightarrow \cA^{1, \cl}(1)$ corresponds to the first Chern class of the line bundle. A connective structure $(\cL, \nabla)$ corresponds to the usual notion of a connection on a line bundle and its curvature corresponds to the usual notion of the curvature.
\end{example}

\begin{example}
The stack $\B^2\Gm$ classifies $\Gm$-gerbes. The morphism $c_1\colon \B^2\Gm\rightarrow \cA^{1, \cl}(2)$ is the usual characteristic class of a gerbe obstructing the existence of a connective structure. For a gerbe $(\cG, \nabla)$ with a connective structure over $X$, the curvature $\curv(\cG, \nabla)\in\cA^{2, \cl}(X, 1)$ is the obstruction of $(\cG, \nabla)$ to admit a curving.
\end{example}

For a morphism $\pi\colon X\rightarrow B$ of derived prestacks we may also talk about connective structures and flat connections on gerbes relative to $\pi$. Namely, for an $n$-gerbe $\cG$ on $X$, a flat connection along $\pi$ is a trivialization of the image of $c_1(\cG)$ under $\cA^{1, \cl}(X, n+1)\rightarrow \cA^{1, \cl}(X/B, n+1)$.

\begin{example}
Suppose $\pi\colon X\rightarrow B$ is a morphism of derived prestacks and $\cG$ is a gerbe on $B$. Then the pullback gerbe $\pi^*\cG$ admits a flat connection relative to $B$. Indeed, the composite
\[\cA^{1, \cl}(B, n+1)\longrightarrow \cA^{1, \cl}(X, n+1)\longrightarrow \cA^{1, \cl}(X/B, n+1)\]
admits a canonical nullhomotopy and $c_1(\pi^*\cG) = \pi^*c_1(\cG)$.
\label{ex:pullbackgerbe}
\end{example}

\subsection{Quasi-coherent sheaves and quasi-coherent stacks}
\label{sect:QCohQStk}

In this section we recall the notion of a quasi-coherent sheaf and a quasi-coherent stack over a derived prestack. To deal with size issues, throughout the paper we fix two universes of small and large sets, so that the set of small sets is large. We denote by $\hCat$ the $\infty$-category of large $\infty$-categories with arbitrary functors as morphisms. We denote by $\PrL_k$ the $\infty$-category of stable presentable $k$-linear $\infty$-categories, so that $\PrL_k\in\hCat$.

\begin{defn}
A derived prestack $X$ is \defterm{small} if it is obtained by a small colimit of derived affine schemes. A morphism $f\colon X\rightarrow Y$ of derived prestacks is \defterm{small} if for every morphism $S\rightarrow Y$ from a derived affine scheme $S$ the base change $X\times_Y S$ is small.
\end{defn}

\begin{example}
Here are two examples of small prestacks.
\begin{enumerate}
\item Suppose $X$ is locally of finite presentation. Then it is obtained by a left Kan extension from a functor $\CAlg^{\leq 0, \mathrm{fp}}\rightarrow \cS$ from the \emph{small} $\infty$-category of finitely presentable connective commutative dg algebras. In particular, it is small.

\item Suppose $X$ is a derived Artin stack. Then it is obtained from derived affine schemes by a sequence of \'{e}tale sheafifications and geometric realizations, both of which are small colimits.
\end{enumerate}
\end{example}

\begin{defn}
Let $X$ be a derived prestack.
\begin{itemize}
\item Suppose $X$ is small. The \defterm{algebra of functions} is the limit
\[\cO(X) = \lim_{\Spec A\rightarrow X} A\in\Mod_k.\]
\item The \defterm{$\infty$-category of quasi-coherent sheaves} is the limit
\[\QCoh(X) = \lim_{\Spec A\rightarrow X} \Mod_A\in\hCat.\]
The structure sheaf $\cO_X\in\QCoh(X)$ is the object given by $A\in\Mod_A$ on each affine.
\item The \defterm{$\infty$-category of sheaves of categories} is the limit
\[\ShvCat(X) = \lim_{\Spec A\rightarrow X} \Mod_{\Mod_A}(\PrL_k)\in\hCat.\]
The categorified structure sheaf $\cQ_X\in\ShvCat(X)$ is the object given by $\Mod_A\in\Mod_{\Mod_A}(\PrL_k)$ on each affine.
\end{itemize}
\end{defn}

For a morphism of derived prestacks $f\colon X\rightarrow Y$ there are natural pullback morphisms and functors
\[f^*\colon \cO(Y)\longrightarrow \cO(X),\qquad f^*\colon \QCoh(Y)\longrightarrow \QCoh(X),\qquad f^*\colon\ShvCat(Y)\longrightarrow \ShvCat(X).\]

If $X$ is small, $\QCoh(X)$ is presentable (see \cite[Proposition 6.2.3.4]{SAG}). Using the adjoint functor theorem we conclude that for a morphism $f\colon X\rightarrow Y$ of small derived prestacks there is a right adjoint
\[f_*\colon\QCoh(X)\longrightarrow \QCoh(Y).\]
Similarly, if $f\colon X\rightarrow Y$ is a small morphism of derived prestacks, the pullback functor on sheaves of categories admits a right adjoint
\[f_*\colon \ShvCat(X)\longrightarrow \ShvCat(Y),\]
see \cite[Section 3.1]{GaitsgoryShvCat} and \cite[Section 10.1.4]{SAG}.

The stacks $\B^{n+1}\Gm$ for small $n$ can be understood as follows:
\begin{itemize}
\item $n=0$. Let $\QCoh^\times$ be the derived stack of invertible quasi-coherent sheaves. We have a natural morphism
\[\B\Gm\times \Z\longrightarrow \QCoh^\times.\]
Namely, we send $m\in \Z$ to $\cO[m]$ and $\Gm$ to automorphisms of $\cO[m]$. Since every invertible $R$-module is Zariski-locally free, the above morphism is an isomorphism (see also \cite[Corollary 2.9.5.7]{SAG}).

\item $n=1$. By \cite[Theorem 5.13]{DAGXI} or \cite[Theorem 2.4]{ToenAzumaya} $\ShvCat$ satisfies \'{e}tale descent, so we have a derived stack $\ShvCat^\times$ of invertible sheaves of categories. Again, we have a natural morphism
\[\B^2\Gm\times \B\Z\longrightarrow \ShvCat^\times\]
to the abelian group stack of invertible elements in $\ShvCat^\sim$. However, not every sheaf of presentable categories is \'{e}tale-locally trivial. It is shown in \cite[Corollary 2.12]{ToenAzumaya} that the image of the above morphism is given by those invertible sheaves of categories which \'{e}tale-locally admit a generator.
\end{itemize}

If $\cL$ is a line bundle on a derived prestack $X$, i.e. a morphism $X\rightarrow \B\Gm$, we may consider the associated invertible sheaf $\cL\in\QCoh(X)$. If $X$ is small, we may define the space of global sections
\[\Gamma(X, \cL) = \Hom_{\QCoh(X)}(\cO_X, \cL).\]
An analogous definition can be given on a categorified level.

\begin{defn}
Let $X$ be a derived prestack and $\cG$ a gerbe which we consider an invertible sheaf of categories over $X$. The $\infty$-category of \defterm{$\cG$-twisted quasi-coherent sheaves} is
\[\QCoh^\cG(X) = \Hom_{\ShvCat(X)}(\cQ_X, \cG).\]
\end{defn}

As before, if $X$ is a small derived prestack, $\QCoh^\cG(X)$ is a presentable $\infty$-category. Moreover, given a morphism $f\colon Y\rightarrow X$ of derived prestacks there is a pullback functor
\[f^*\colon \QCoh^\cG(X)\longrightarrow \QCoh^{f^*\cG}(Y).\]
If $X$ and $Y$ are small, it admits a right adjoint
\[f_*\colon \QCoh^{f^*\cG}(Y)\longrightarrow \QCoh^\cG(X).\]

\subsection{Prequantization}
\label{sect:prequantization}

We are ready to define prequantization of shifted symplectic stacks. Throughout we assume $n\geq 0$ is a natural number.

\begin{defn}
Let $(X, \omega_X)$ be an $n$-shifted symplectic stack. Its \defterm{prequantization} is an $n$-gerbe $(\cG, \nabla)$ with a connective structure on $X$ together with an equivalence $\curv(\cG, \nabla) = \omega_X$ in $\cA^{2, \cl}(X, n)$.
\label{def:prequantumsymplectic}
\end{defn}

Let us make the following observation.

\begin{prop}
Let $\pi\colon X\rightarrow B$ be a morphism of derived prestacks, $\cG$ an $n$-gerbe on $B$ together with a connective structure $\nabla$ on $f^*\cG$ extending the canonical flat connection on $\pi^*\cG$ relative to $B$. Then $\pi$ has a natural structure of an $n$-shifted isotropic fibration.
\label{prop:prequantumisotropicfibration}
\end{prop}
\begin{proof}
The curvature $\curv(\pi^*\cG, \nabla)$ gives an $n$-shifted presymplectic structure on $X$. Its image in $\cA^{2, \cl}(X/B)$ has a canonical nullhomotopy since the connective structure along the fibers of $\pi$ is flat.
\end{proof}

Therefore, we may give the following definition.

\begin{defn}
Let $\pi\colon X\rightarrow B$ be a morphism of derived Artin stacks locally of finite presentation. A \defterm{prequantum $n$-shifted Lagrangian fibration} is given by the following data:
\begin{itemize}
\item An $n$-gerbe $\cG$ on $B$.
\item An extension of the natural relative flat connection on $\pi^*\cG$ to a connective structure $\nabla$.
\end{itemize}
These have to satisfy the nondegeneracy condition that the induced $n$-shifted isotropic fibration given by \cref{prop:prequantumisotropicfibration} is an $n$-shifted Lagrangian fibration. In this case we say $\cG$ and $\nabla$ define a \defterm{prequantization} of the $n$-shifted Lagrangian fibration.
\label{def:prequantumfibration}
\end{defn}

Let us now incorporate Lagrangians in these definitions. The notion of a prequantization of a single Lagrangian is a direct analog of \cref{def:prequantumsymplectic}.

\begin{defn}
Let $f\colon L\rightarrow X$ be a morphism of derived Artin stacks locally of finite presentation. A \defterm{prequantum $n$-shifted Lagrangian structure} is given by an $n$-gerbe $(\cG, \nabla)$ with a connective structure on $X$ together with a trivialization of $f^*(\cG, \nabla)$, such that the induced $n$-shifted presymplectic structure on $X$ and the induced $n$-shifted isotropic structure on $f$ are nondegenerate.
\label{def:prequantumLagrangian}
\end{defn}

\begin{remark}
Recall the notion of an $n$-shifted Lagrangian triple $X\xleftarrow{f} Z\xrightarrow{g} B$ from \cref{def:Lagrangiantriple}. Its \defterm{prequantization} is given by a prequantization $(\cG_X, \nabla)$ of $X$, an $n$-gerbe $\cG_B$ on $B$ and an isomorphism $f^*(\cG_X, \nabla)\cong (g^*\cG_B, \nabla_{/B})$ ($\nabla/B$ is the natural flat connection on $g^*\cG_B$ relative to $B$) of $n$-gerbes with a connective structure relative to $B$ such that applying $\curv$ we obtain the Lagrangian structure on $Z\rightarrow X\times B$ relative to $B$. For $Z=X$ this reduces to \cref{def:prequantumfibration} and for $B=\pt$ this reduces to \cref{def:prequantumLagrangian}.
\end{remark}

Suppose $\pi_X\colon X\rightarrow B_X$ is a prequantum $n$-shifted Lagrangian fibration and consider a commutative diagram
\[
\xymatrix{
L \ar^{f}[r] \ar^{\pi_L}[d] & X \ar^{\pi_X}[d] \\
B_L \ar^{g}[r] & B_X
}
\]
By definition $\pi_X^*\cG$ has a connective structure $\nabla$ extending the natural flat connection along $\pi_X$. Therefore, $\pi_L^*g^*\cG$ has a connective structure $f^*\nabla$ extending the natural flat connection along $\pi_L$. Let us in addition fix a trivialization of the $n$-gerbe $g^*\cG$ on $B_L$. Then we obtain a natural one-form $f^*\nabla\in\cA^1(L, n)$ which vanishes along the fibers of $L\rightarrow B_L$. Moreover, by construction
\[f^*\curv(\pi_X^*\cG, \nabla) = \ddr(f^*\nabla)\in\cA^{2,\cl}(L, n).\]

\begin{defn}
Consider a diagram
\begin{equation}
\xymatrix{
L \ar^{f}[r] \ar^{\pi_L}[d] & X \ar^{\pi_X}[d] \\
B_L \ar^{g}[r] & B_X
}
\label{eq:prequantumrelativeLagrangianfibration}
\end{equation}
of derived Artin stacks locally of finite presentation. A \defterm{prequantum $n$-shifted Lagrangian fibration} is given by the following data:
\begin{enumerate}
\item An $n$-gerbe $\cG$ on $B_X$.
\item An extension of the natural relative flat connection on $\cG$ along $X\rightarrow B_X$ to a connective structure $\nabla$ on the whole of $X$.
\item A trivialization of the $n$-gerbe $g^*\cG$ on $B_L$.
\item A trivialization of the one-form
\[f^*\nabla\in\fib(\cA^1(L, n)\longrightarrow \cA^1(L/B_L, n)).\]
\end{enumerate}
These have to satisfy the following nondegeneracy conditions:
\begin{enumerate}
\item The induced $n$-shifted isotropic fibration structure on $X\rightarrow B_X$ is nondegenerate.

\item The induced $n$-shifted isotropic fibration structure on \eqref{eq:prequantumrelativeLagrangianfibration} is nondegenerate.
\end{enumerate}
In this case we say the above data define a \defterm{prequantization} of the $n$-shifted Lagrangian fibration structure on \eqref{eq:prequantumrelativeLagrangianfibration}.
\end{defn}

We will define prequantizations of $n$-shifted Lagrangian correspondences $X\leftarrow L\rightarrow Y$ to be diagrams
\[
\xymatrix@R=0cm{
& L \ar[dl] \ar[dd]\ar[dr] & \\
X \ar[dd] && Y \ar[dd] \\
& B_L \ar[dl] \ar[dr] & \\
B_X && B_Y
}
\]
of stacks together with $n$-gerbes $\cG_X, \cG_Y$ on $B_X$ and $B_Y$ with an isomorphism of their pullbacks to $B_L$ and connective structures on their pullbacks to $X$ and $Y$ such that
\[
\xymatrix{
L \ar[d] \ar[r] & \overline{X}\times Y \ar[d] \\
B_L \ar[r] & B_X\times B_Y
}
\]
is a prequantum $n$-shifted Lagrangian fibration.

\begin{example}
Consider the diagram \eqref{eq:prequantumrelativeLagrangianfibration}, where $X=B_X=\pt$ is equipped with the trivial prequantization. A trivialization of the trivial $n$-gerbe on $B_L$ is simply an $(n-1)$-gerbe $\cG_L$ on $B_L$. The corresponding one-form
\[f^*\nabla\in\fib(\cA^1(L, n)\longrightarrow \cA^1(L/B_L, n))\]
is the pullback under $\pi_L$ of $c_1(\cG_L)$. In other words, in this case we simply recover the notion of a prequantum $(n-1)$-shifted Lagrangian fibration on $L\rightarrow B_L$.
\end{example}

\begin{example}
Consider the diagram \eqref{eq:transverseLagrangian} from \cref{ex:transverseLagrangian}, where $L\rightarrow X\rightarrow B_X$ is an isomorphism. In this case the $n$-gerbe $\cG$ is trivial. The connective structure $\nabla$ on $X$ corresponds to a one-form
\[\alpha\in\fib(\cA^1(X, n)\longrightarrow \cA^1(X/B, n)).\]
The additional data is a trivialization of the pullback one-form
\[f^*\alpha\in\fib(\cA^1(L, n)\longrightarrow \cA^1(L/B, n)).\]
\end{example}

\begin{example}
Recall that the de Rham stack $X_{\dR}$ of any prestack has the trivial cotangent complex. In particular, it can be regarded as an $n$-shfited symplectic stack for any $n$. Now suppose $X\rightarrow B$ is a prequantum $n$-shifted Lagrangian fibration and consider the diagram
\[
\xymatrix{
X \ar[r] \ar[d] & X_{\dR} \ar[d] \\
B \ar[r] & B_{\dR}
}
\]
It has a natural structure of a prequantum $(n+1)$-shifted Lagrangian fibration given by the trivial data on $X_{\dR}\rightarrow B_{\dR}$ whose trivialization on $X\rightarrow B$ is specified by the data of the prequantum $n$-shifted Lagrangian fibration.
\label{ex:prequantumDR}
\end{example}

We have the following prequantum analog of \cref{thm:Lagrangianfibrationintersection}.

\begin{thm}
Let $X\rightarrow B_X$ be a prequantum $n$-shifted Lagrangian fibration together with the prequantum data on the squares in the diagram
\[
\xymatrix@R=0cm{
L_1 \ar[dr] \ar[dd] && L_2 \ar[dl] \ar[dd] \\
& X \ar^{\pi_X}[dd] & \\
B_{L_1} \ar[dr] && B_{L_2} \ar[dl] \\
& B_X &
}
\]
Then passing to pullbacks we obtain a prequantum $(n-1)$-shifted Lagrangian fibration
\[\pi\colon L_1\times_X L_2\rightarrow B_{L_1}\times_{B_X} B_{L_2}.\]
\label{thm:prequantumLagrangianfibrationintersection}
\end{thm}
\begin{proof}
We are given an $n$-gerbe $\cG_X$ on $B_X$ together with its trivializations along $B_{L_i}\rightarrow B_X$. Therefore, we obtain a natural $(n-1)$-gerbe $\cG_{L_1, L_2}$ on $B_{L_1}\times_{B_X} B_{L_2}$. Let us now 

The proof of \cite[Theorem 2.9]{PTVV} constructs a map from the limit of
\[
\xymatrix{
0 \ar[dr] && \cA^{1, \cl}(X, n+1) \ar[dl] \ar[dr] && 0 \ar[dl] \\
& \cA^{1, \cl}(L_1, n+1) && \cA^{1, \cl}(L_2, n+1)
}
\]
to $\cA^{1, \cl}(L_1\times_X L_2, n)$. It is easy to see that the image of $c_1(\pi_X^*\cG_X)$ under this map is exactly $c_1(\pi^*\cG_{L_1, L_2})$. In particular, it shows that the connective structure on $\pi_X^*\cG_X$ together with its trivializations on $L_1$ and $L_2$ give rise to a connective structure on $\cG_{L_1, L_2}$. Working relatively to the bases $B_X, B_{L_1}$ and $B_{L_2}$ we see that it is moreover compatible with the natural flat connections along the fibers.

Finally, the nondegeneracy condition for the induced isotropic fibration follows from \cref{thm:Lagrangianfibrationintersection}.
\end{proof}

\subsection{Twisted cotangent bundles}

In this section we study prequantizations of shifted twisted cotangent bundles. Let us begin with the case of the usual shifted cotangent bundle.

Let $X$ be a derived Artin stack locally of finite presentation. By construction there is a natural Liouville one-form
\[\lambda\in\fib(\cA^1(\T^*[n] X, n)\longrightarrow \cA^1(X, n))\]
so that the $n$-shifted symplectic structure on $\T^*[n] X$ is given by $\ddr \lambda$.

\begin{prop}
Consider the projection $\pi\colon \T^*[n] X\rightarrow X$. The trivial $n$-gerbe $\cG$ on $X$ together with a connective structure on $\pi^*\cG$ given by the Liouville one-form $\lambda$ defines a prequantum $n$-shifted Lagrangian fibration structure on $\pi$.
\label{prop:prequantumcotangent}
\end{prop}
\begin{proof}
The only nontrivial fact is the nondegeneracy of the prequantum data which is proven in \cite[Theorem 2.2]{CalaqueCotangent}.
\end{proof}

We have a relative analog of the above statement. Let $g\colon Y\rightarrow X$ be a morphism of derived Artin stacks locally of finite presentation. Consider the diagram
\begin{equation}
\xymatrix{
\N^*[n] Y \ar^{\pi_Y}[d] \ar^{f}[r] & \T^*[n]X \ar^{\pi_X}[d] \\
Y \ar^{g}[r] & X
}
\label{eq:conormalfibration}
\end{equation}
The pullback of the Liouville one-form
\[f^*\lambda\in\fib(\cA^1(\N^*[n]Y, n)\rightarrow \cA^1(\N^*[n]Y/Y, n))\]
has a natural nullhomotopy. Combining with \cite[Proposition 2.12]{CalaqueCotangent} we obtain the following statement.

\begin{prop}
The trivial $n$-gerbe on $X$ with the naive trivialization of its pullback on $Y$ equipped with the connective structures on $\T^*[n]X$ and $\N^*[n] Y$ given by the Liouville one-form define the structure of a prequantum $n$-shifted Lagrangian fibration on \eqref{eq:conormalfibration}.
\label{prop:cotangentprequantization}
\end{prop}

Let us now consider the twisted case. Suppose $\cG$ is an $n$-gerbe on $X$ and $\alpha\in\cA^{1, \cl}(X, n+1)$. Consider the diagram
\[
\xymatrix{
X \ar^-{\Gamma_\alpha}[r] \ar^{\pi_L = \id}[d] & \T^*[n+1] X \ar^{\pi_X}[d] \\
X \ar^{g=\id}[r] & X
}
\]

Consider the trivial $(n+1)$-gerbe $\cG_X$ on $X$ with the connective structure on $\pi_X^*\cG_X$ given by the Liouville one-form $\lambda$. Consider the trivialization of the $(n+1)$-gerbe $g^*\cG_X$ given by $\cG$. Then
\[\alpha - c_1(\cG) = f^*\nabla\in\cA^1(X, n+1).\]
In particular, if we assume that $\alpha=c_1(\cG)$ is the characteristic class of $\cG$, the above element is canonically zero. We may summarize this discussion in the following statement.

\begin{prop}
Suppose $\cG$ is an $n$-gerbe on $X$. Then the diagram
\[
\xymatrix{
X \ar^-{\Gamma_{c_1(\cG)}}[r] \ar@{=}[d] & \T^*[n+1] X \ar[d] \\
X \ar@{=}[r] & X
}
\]
has a natural structure of a prequantum $(n+1)$-shifted Lagrangian fibration.
\label{prop:prequantumcotangentsection}
\end{prop}

We are ready to formulate a large class of prequantum $n$-shifted Lagrangian fibrations.

\begin{thm}
Suppose $X$ is a derived Artin stack locally of finite presentation and $\cG$ an $n$-gerbe on $X$. Then
\[\T^*_{c_1(\cG)}[n] X\longrightarrow X\]
has a natural structure of a prequantum $n$-shifted Lagrangian fibration determined by $\cG$.
\label{thm:prequantumtwistedcotangent}
\end{thm}
\begin{proof}
Consider the diagram
\[
\xymatrix@R=0cm{
X \ar^-{\Gamma_{c_1(\cG)}}[dr] \ar@{=}[dd] && X \ar_-{\Gamma_0}[dl] \ar@{=}[dd] \\
& \T^*[n+1]X \ar[dd] & \\
X \ar@{=}[dr] && X \ar@{=}[dl] \\
& X &
}
\]
By \cref{prop:prequantumcotangentsection} each square has the structure of a prequantum $(n+1)$-shifted Lagrangian fibration. Therefore, by \cref{thm:prequantumLagrangianfibrationintersection} the induced map
\[\T^*_{c_1(\cG)}[n] X\longrightarrow X\]
has the structure of a prequantum $n$-shifted Lagrangian fibration.
\end{proof}

Let $f\colon Y\rightarrow X$ be a morphism of derived Artin stacks locally of finite presentation together with an $n$-gerbe $\cG$ on $X$ together with the trivialization of $f^*\cG$ on $Y$. Then one can similarly prove that the diagram
\[
\xymatrix{
\N^*_{c_1(\cG)}[n] Y \ar[r] \ar[d] & \T^*_{c_1(\cG)}[n] X \ar[d] \\
Y \ar[r] & X
}
\]
has the structure of a prequantum $n$-shifted Lagrangian fibration.

\subsection{Classifying stack}

Let $G$ be a reductive algebraic group and $\langle -, -\rangle$ is a nondegenerate invariant symmetric bilinear pairing on its Lie algebra $\g$. Then the classifying stack $\B G$ has a natural $2$-shifted symplectic structure. It has the following explicit form.

Consider the epimorphism $\pt\rightarrow \B G$. By descent
\[\cA^{2, \cl}(\B G, 2) \cong \lim_k \cA^{2, \cl}(G^k, 2),\]
where on the right we consider the totalization associated to the simplicial scheme
\[
\xymatrix{
\pt  & G \ar@<.5ex>[l] \ar@<-.5ex>[l]  & G\times G \ar@<.8ex>[l] \ar[l] \ar@<-.8ex>[l] & \ldots \ar@<1.5ex>[l] \ar@<.5ex>[l] \ar@<-.5ex>[l] \ar@<-1.5ex>[l]
}
\]

Since $G$ is affine, we see that a 2-shifted symplectic structure on $\B G$ corresponds to a two-form $\omega$ on $G\times G$ and a three-form $H$ on $G$ which together satisfy
\begin{align*}
m_{23}^*\omega + p_{23}^* \omega &= m_{12}^*\omega + p_{12}^*\omega \\
m^* H &= p_1^* H + p_2^* H + \ddr \omega \\
\ddr H &= 0
\end{align*}
where $p$ and $m$ denote projection and multiplication maps between a number of copies of $G$. Let $\theta, \overline{\theta}\in\Omega^1(G; \g)$ be the left and right Maurer--Cartan forms. It is shown in \cite{SafronovCS} that the 2-shifted symplectic structure on $\B G$ is represented by
\[\omega = \frac{1}{2}\langle p_1^*\theta, p_2^*\overline{\theta}\rangle\in\Omega^2(G\times G),\qquad H = \frac{1}{12}\langle\theta, [\theta, \theta]\rangle\in\Omega^3(G).\]

Let us first observe that we should not expect an existence of a Lagrangian fibration on $\B G$.

\begin{prop}
Suppose $X\rightarrow B$ is an $n$-shifted Lagrangian fibration, where $n$ is even. Then the virtual dimension $\dim(X)$ is even.
\label{prop:evendimension}
\end{prop}
\begin{proof}
Consider the fiber sequence
\[\bT_{X/B}\longrightarrow \bT_X\longrightarrow \bL_{X/B}[n]\]
in $\QCoh(X)$. We have $\dim(\bL_{X/B}[n]) = \dim(\bL_{X/B}) = \dim(\bT_{X/B})$, where the first equality uses the fact that $n$ is even. By additivity of the Euler characteristic we obtain
\[\dim(X) = 2\dim(\bT_{X/B}).\]
\end{proof}

\begin{cor}
The $2$-shifted symplectic stack $\B\SL_2$ does not admit a Lagrangian fibration.
\end{cor}
\begin{proof}
Indeed, $\dim(\B\SL_2) = -\dim(\sl_2) = -3$.
\end{proof}

Since both gerbes and differential forms satisfy \'{e}tale descent, so do gerbes with connective structures. Thus, a prequantization of $\B G$ is given by the following data:
\begin{itemize}
\item A gerbe with a connective structure $(\cG, \nabla_1)$ on $G$ with $\curv(\cG, \nabla) = H$.

\item A trivialization $(\cL, \nabla_2)$ of $m^*(\cG, \nabla_1)\otimes p_1^*(\cG, \nabla_1)^{-1}\otimes p_2^*(\cG, \nabla_2)^{-1}$ on $G\times G$ with $\curv(\cL, \nabla_2) = \omega$.

\item An isomorphism $m_{23}^*(\cL, \nabla_2)\otimes p_{23}^*(\cL, \nabla_2)\cong m_{12}^*(\cL, \nabla_2)\otimes p_{12}^*(\cL, \nabla_2)$ on $G^3$.

\item A coherence condition for the above isomorphisms on $G^4$.
\end{itemize}

\begin{prop}
Suppose $G$ is a split connected reductive group. The $2$-shifted symplectic stack $\B G$ does not admit a prequantization.
\label{prop:BGprequantization}
\end{prop}
\begin{proof}
Since $G$ is smooth, by \cite[Proposition 1.4]{GrothendieckBrauer} every $\Gm$-gerbe $\cG$ on $G$ is torsion. In particular, $c_1(\cG) = 0\in\cA^{1, \cl}(G, 2)$. Therefore, a connective structure $\nabla$ on $\cG$ corresponds to a one-form $\alpha\in\cA^1(G, 1)$ of degree 1. Since $G$ is affine, $\alpha$ is nullhomotopic. Therefore, $\curv(\cG, \nabla)$ is nullhomotopic in $\cA^{2, \cl}(G, 1)$. Since $G$ is affine, we also have $\pi_0(\cA^{2, \cl}(G, 1)) = \H^3_{\dR}(G)$.

Let us now assume that $G$ is semisimple. It is well-known that in this case $H\in\H^3_{\dR}(G)$ represents a nonzero cohomology class, so there does not exist a gerbe $(\cG, \nabla)$ with a connective structure whose curvature is $H$.

If $G$ is reductive, a nondegenerate invariant symmetric bilinear form $\langle -, -\rangle$ restricts to such on the derived subgroup. So, if the derived subgroup is nontrivial, the gerbe $(\cG, \nabla)$ does not exist either.

Let us assume that the derived group is trivial, i.e. $G$ is a torus. In this case $H = 0$ and $\omega$ is a closed multiplicative two-form on $G\times G$. Since $\cG$ is torsion, a power of $(\cL, \nabla_2)$ is a line bundle with a connection. Since $\Pic(G) = 0$, $\curv(\cL, \nabla_2)$ is exact. Let us assume $G=\Gm$ for simplicity. Then $\omega = \frac{\ddr x}{x}\wedge \frac{\ddr y}{y}$ for $x,y$ the natural coordinates on $\Gm\times \Gm$. In particular, it is not exact.
\end{proof}

Assume $\langle -, -\rangle$ is integral, i.e. it comes from an integral Weyl-invariant quadratic form on the cocharacter lattice. If we regard $G$ as a complex Lie group, a prequantization exists. Indeed, Brylinski and Deligne \cite{BrylinskiDeligne} construct a multiplicative $\cK_2$-torsor on $G$ associated to $\langle-, -\rangle$, where $\cK_2$ is the Zariski sheafification of the Quillen's $K_2$ functor. Applying the Beilinson regulator \cite{Beilinson,BrylinskiRegulator,BrylinskiGerbes} one obtains a multiplicative gerbe with a connective structure.

In the setting of smooth manifolds Waldorf \cite{Waldorf} has constructed the corresponding multiplicative gerbe with a connective structure over any compact simple simply-connected Lie group. The corresponding 2-gerbe over $\B G$ is known as the \emph{Chern--Simons 2-gerbe} \cite{CJMSW}.

\subsection{$(-1)$-shifted prequantization}

Let $X$ be a $(-1)$-shifted symplectic stack. In \cref{sect:prequantization} we have defined prequantizations of $n$-shifted symplectic stacks for $n\geq 0$. The case $n=-1$ is implicitly given by \cref{def:prequantumLagrangian} since a $(-1)$-shifted symplectic stack $X$ gives rise to a $0$-shifted Lagrangian structure on the canonical projection $X\rightarrow \pt$, where we view $\pt$ as a $0$-shifted symplectic stack in the obvious way. Let us unpack this definition.

\begin{itemize}
\item A $(-1)$-gerbe is the same as an invertible function $f_0\colon X\rightarrow \Gm$.

\item A connective structure on a $(-1)$-gerbe is given by a one-form $h_1$ on $X$ of degree $-1$ such that
\begin{equation}
\ddr f_0 / f_0 + \d h_1 = 0.
\label{eq:m1connection1}
\end{equation}
Its curvature is $\ddr h_1\in\cA^{2, \cl}(X, -1)$.

\item A homotopy between the above curvature and a given $(-1)$-shifted symplectic structure $\omega = \omega_2 + \omega_3 + \dots$ is given by a collection $\{h_2, h_3, \dots\}$ of forms, where $h_p$ is a $p$-form of degree $1-p$ which satisfy the equations
\begin{align}
\ddr h_1 + \omega_2 + \d h_2 &=  0 \nonumber \\
\ddr h_2 + \omega_3 + \d h_3 &= 0 \label{eq:m1connection2} \\
\dots \nonumber
\end{align}
\end{itemize}

Let us introduce the notation
\[f_0 + f_1 + \dots = f_0\exp(\sum h_i).\]
Then the equations \eqref{eq:m1connection1} and \eqref{eq:m1connection2} can be written as
\[(\d + \ddr + \omega\wedge)(f_0 + f_1 + \dots) = 0.\]
Let us rephrase it in the following way. Recall from \cite[Section 1.3]{CPTVV} that given a graded complex $(A = \oplus_n A(n), \d)$ with a mixed structure $\epsilon$ its \emph{realization} is
\[|A| = \prod_{n\geq 0} A(n)\]
with the differential $\d + \epsilon$. The graded commutative algebra of differential forms $\DR(X)$ admits a (weak) mixed structure given by $\epsilon = \ddr + \omega\wedge(-)$, which is square-zero precisely because $\omega$ is a closed form. In particular, we may consider its realization in the above sense.

\begin{remark}
The differential $\epsilon = \ddr + \omega\wedge(-)$ splits into components which increase weights by positive integers. This is called a weak mixed structure in \cite{CPTVV}.
\end{remark}

\begin{defn}
Let $X$ be a derived Artin stack locally of finite presentation and $\alpha\in\cA^{1, \cl}(X, 0)$ a closed one-form of degree $0$. The \defterm{$\alpha$-twisted de Rham complex} is the complex
\[\Omega_\alpha(X) = |(\DR(X), \ddr + \alpha\wedge(-))|.\]
\end{defn}

Note that by construction there is a natural projection $\Omega_\alpha(X)\rightarrow \cO(X)$.

\begin{example}
Let $X$ be a smooth variety and $f\colon X\rightarrow \bA^1$ a function. Then $\ddr f\in\cA^{1, \cl}(X, 0)$ is a closed one-form of degree 0. Then $\Omega_{\ddr f}(X)$ is the usual twisted de Rham complex with the differential $\ddr + (\ddr f)\wedge(-)$.
\end{example}

We can summarize the above discussion as follows.

\begin{prop}
Let $(X, \omega)$ be a $(-1)$-shifted symplectic stack. A prequantization of $X$ is the same as a closed element $f\in\Omega_\omega(X)$ whose projection to $\cO(X)$ is invertible.
\end{prop}

We may similarly rephrase the notion of a prequantization of a $0$-shifted Lagrangian morphism $L\rightarrow X$. Let $X$ be a derived Artin stack locally of finite presentation. Recall (e.g. see \cite[Section 7]{GRCrys}) that for any closed one-form $\alpha\in\cA^{1, \cl}(X, 1)$ on $X$ of degree $1$ one can define the $\infty$-category $\cD_\alpha(X)$ of $\alpha$-twisted $D$-modules. A line bundle with a connection $(\cL, \nabla)$ on $X$ with curvature $\curv(\cL, \nabla) = \omega$ gives rise to an object $(\cL, \nabla)\in\cD_\alpha(X)$.

\begin{prop}
Let $(X, \omega)$ be a $0$-shifted symplectic stack. A prequantization of $X$ is the same as an object $(\cL, \nabla)\in\cD_\omega(X)$ whose image in $\QCoh(X)$ is an invertible sheaf concentrated in degree 0.
\end{prop}

We will now discuss a relationship between $(-1)$-shifted prequantization and the BV formalism. From now on we denote by $Y$ a smooth scheme. Consider $\Gamma(Y, \Sym(\T_Y[1]))$, the algebra of polyvector fields. There is an obvious isomorphism
\[\Gamma(Y, \Sym(\Omega^1_Y[-1]))\cong \Gamma(Y, \Sym(\T_Y[1])\otimes K_Y)[-\dim Y]\]
under which the de Rham differential $\ddr$ goes to the divergence operator $\dive$. Let $X = \T^*[-1] Y$ and $\pi\colon X\rightarrow Y$ the projection. Then
\[\Gamma(Y, \Sym(\T_Y[1])\otimes K_Y)[-\dim Y]\cong \Gamma(X, \pi^* K_Y)[-\dim Y].\]
Under this isomorphism the divergence operator $\dive$ goes to the BV Laplacian $\Delta$ \cite{WittenBV}.

\begin{example}
Choose \'{e}tale coordinates $\{x_1, \dots, x_n\}$ on an open subset of $Y$. Let $\{p_1, \dots, p_n\}$ be the dual coordinates along the fibers of $\T^*[-1] Y\rightarrow Y$. Let $\omega=\ddr x_1\wedge \dots \ddr x_n$. Then
\[\Delta(f(p, x)\omega) = \sum_{i=1}^n \frac{\partial^2 f}{\partial p_i \partial x_i} \omega.\]
\end{example}

\begin{prop}
Let $\omega$ be the standard $(-1)$-shifted symplectic structure on $\T^*[-1] Y$. Then we have quasi-isomorphisms
\[\Omega_\omega(\T^*[-1] Y)\cong \Omega(Y)\cong (\Gamma(\T^*[-1]Y, \pi^* K_Y)[-\dim Y], \Delta).\]
\label{prop:BVquantization}
\end{prop}
\begin{proof}
We have already explained the last isomorphism, so we just need to construct the first quasi-isomorphism. The projection and the zero section $\T^*[-1] Y\rightleftarrows Y$ induce maps
\[
\xymatrix{
\Gamma(Y, \Sym(\Omega^1_Y[-1])) \ar@<.5ex>^-{i}[r] & \Gamma(\T^*[-1] Y, \widehat{\Sym}(\bL_{\T^*[-1] Y}[-1])) \ar@<.5ex>^-{p}[l]
}
\]
such that $p\circ i = \id$. The maps $p$ and $i$ intertwine the de Rham differentials $\ddr$ on both sides. Let $e$ be the Euler vector field along the fibers of $\T^*[-1] Y\rightarrow Y$ and let $\iota_e$ be the contraction. Then
\[\cL_e = \ddr \iota_e + \iota_e \ddr\]
defines a nonnegative grading on $\Gamma(\T^*[-1] Y, \widehat{\Sym}(\bL_{\T^*[-1] Y}[-1]))$ whose degree zero component is the algebra of differential forms $\Gamma(Y, \Sym(\Omega^1_Y[-1]))$. This constructs a homotopy $h$, such that
\[ip - \id = \ddr h + h\ddr.\]
In this way we obtain a special deformation retract
\[
\xymatrix{
(\Gamma(Y, \Sym(\Omega^1_Y[-1])), \ddr) \ar@<.5ex>^-{i}[r] & (\Gamma(\T^*[-1] Y, \widehat{\Sym}(\bL_{\T^*[-1] Y}[-1])), \ddr) \ar@<.5ex>^-{p}[l]
}
\]
Observe that the operator $\omega\wedge(-)\circ \iota_e$ on $\Gamma(\T^*[-1] Y, \widehat{\Sym}(\bL_{\T^*[-1] Y}[-1]))$ is nilpotent since it increases the form degree along $Y$. Note that $p\circ\omega\wedge(-)\circ\iota_e\circ\omega\wedge(-)\circ i = 0$ since the Euler vector field vanishes along the zero section. Therefore, by the homological perturbation lemma \cite{Crainic} there is a special deformation retract
\[
\xymatrix{
(\Gamma(Y, \Sym(\Omega^1_Y[-1])), \ddr) \ar@<.5ex>^-{i_1}[r] & (\Gamma(\T^*[-1] Y, \widehat{\Sym}(\bL_{\T^*[-1] Y}[-1])), \ddr +\wedge\wedge(-)) \ar@<.5ex>^-{p_1}[l]
}
\]
\end{proof}

There is another appearance of the BV Laplacian in a similar context explained in \cite{Severa}. The de Rham complex we consider
\[(\Omega(\T^*[-1] Y), \ddr) = (\Gamma(Y, \widehat{\Sym}(\bL_{\T^*[-1] Y}[-1])), \ddr)\]
involves completed differential forms. Let
\[\Omega^{\pol}(\T^*[-1] Y)\subset \Omega(\T^*[-1] Y)\]
be the subspace of polynomial differential forms. We may also consider the twisted version $\Omega^{\pol}_\omega(\T^*[-1] Y)$ and $\Omega^{\pol, \lambda}_\omega(\T^*[-1] Y)$ the de Rham complex with the differential $\lambda\ddr + \omega\wedge(-)$.

Consider the inclusion and projection maps
\[
\xymatrix{
\Gamma(\T^*[-1] Y, \pi^* K_Y)[-\dim Y]\ar@<.5ex>^-{i}[r] & \Gamma(\T^*[-1] Y, \Sym(\bL_{\T^*[-1] Y}[-1])) \ar@<.5ex>^-{p}[l]
}
\]
Let $\iota_\pi$ be the contraction with the Poisson bivector on $\T^*[-1] Y$. Then $\omega\wedge(-)\circ\iota_\pi + \iota_\pi\circ\omega\wedge(-)$ defines a nonpositive grading with $\Gamma(\T^*[-1] Y, \pi^* K_Y)[-\dim Y]$ the zeroth component. The operator $(\id - \ddr \iota_\pi)$ is invertible on $\Gamma(\T^*[-1] Y, \Sym(\bL_{\T^*[-1] Y}[-1]))$, so we may again apply the homological perturbation lemma. The induced differential on $\Gamma(\T^*[-1] Y, \pi^* K_Y)[-\dim Y]$ is exactly the BV Laplacian.

\begin{thm}[Severa]
Let $Y$ be a smooth scheme. For any $\lambda$ there is a quasi-isomorphism
\[\Omega^{\pol, \lambda}_\omega(\T^*[-1] Y)\cong (\Gamma(\T^*[-1] Y, \pi^* K_Y)[-\dim Y], \lambda\Delta).\]
\end{thm}

\begin{remark}
The operator $(\id - \ddr \iota_\pi)$ is \emph{not} invertible on the completed algebra of differential forms $\Gamma(\T^*[-1] Y, \widehat{\Sym}(\bL_{\T^*[-1] Y}[-1]))$. In particular, in this case we cannot apply the homological perturbation lemma to the same setup.
\end{remark}

For a QP-manifold $X$, the \defterm{quantum action functional} $\exp(S)$ is a semidensity annihilated by the BV Laplacian \cite{Schwarz}. We see that in the context of shifted symplectic geometry $\Omega_\omega(X)$ plays the role of the complex of semidensities.

\section{Geometric quantization}

In this section we finally define the geometric quantization of an $n$-shifted Lagrangian fibration.

\subsection{Definition}

\begin{defn}
Let $\pi\colon X\rightarrow B$ be a prequantum $1$-shifted Lagrangian fibration specified, in particular, by a gerbe $\cG$ on $B$. The \defterm{geometric quantization} of $\pi$ is the $\infty$-category $\QCoh^\cG(B)$.
\label{def:geometricquantization1stack}
\end{defn}

\begin{example}
Suppose $Y$ is a derived Artin stack locally of finite presentation and $\cG$ a gerbe on $Y$. By \cref{thm:prequantumtwistedcotangent} the projection $\T^*_{c_1(\cG)}[1]Y\rightarrow Y$ has a natural prequantization. Its geometric quantization is $\QCoh^\cG(Y)$.
\end{example}

\begin{defn}
Let
\[
\xymatrix@R=0cm{
& L \ar[dl] \ar[dd]\ar[dr] & \\
X \ar[dd] && Y \ar[dd] \\
& B_L \ar^{f}[dl] \ar_{g}[dr] & \\
B_X && B_Y
}
\]
be a prequantum $1$-shifted Lagrangian correspondence specified, in particular, by gerbes $\cG_X$ and $\cG_Y$ on $X$ and $Y$ together with an isomorphism $f^*\cG_X\cong g^*\cG_Y$. The \defterm{geometric quantization} of this $1$-shifted Lagrangian correspondence is the functor
\[\QCoh^{\cG_X}(B_X)\xrightarrow{f^*}\QCoh^{f^*\cG_X}(B_L)\cong\QCoh^{g^*\cG_Y}(B_L)\xrightarrow{g_*}\QCoh^{\cG_Y}(B_Y).\]
\label{def:geometricquantization1correspondence}
\end{defn}

\begin{example}
If $L\rightarrow X$ is a $1$-shifted Lagrangian, we may interpret it as a $1$-shifted Lagrangian correspondence in two ways: either as $\pt\leftarrow L\rightarrow X$ or as $X\leftarrow L\rightarrow \pt$. In the first interpretation the geometric quantization of a prequantum 1-shifted Lagrangian $L$ gives rise to a functor $\Ch\rightarrow \QCoh^\cG(X)$, i.e. to an object of $\QCoh^\cG(X)$. In the second interpretation it gives rise to a functor $\QCoh^\cG(X)\rightarrow \Ch$. If $L$ is a smooth and proper scheme, we may adjust the prequantum data by the canonical bundle of $L$, so that the second functor is left adjoint to the first functor.
\end{example}

\begin{example}
Let $f\colon Y\rightarrow X$ be a morphism of derived Artin stacks locally of finite presentation, $\cG$ a gerbe on $X$ together with a trivialization of $f^*\cG$. Then
\[
\xymatrix{
\N^*_{c_1(\cG)}[1] Y \ar[r] \ar[d] & \T^*_{c_1(\cG)}[1] X \ar[d] \\
Y \ar[r] & X
}
\]
has the structure of a prequantum $1$-shifted Lagrangian fibration. Its geometric quantization gives rise to an object $f_*\cQ_Y\in\QCoh^\cG(X)$ as well as a functor $\QCoh^\cG(X)\rightarrow \Ch$ given by $\cE\mapsto \Gamma(Y, f^*\cE)$.
\end{example}

\begin{example}
Let $X\rightarrow B$ be a prequantum $0$-shifted Lagrangian fibration specified, in particular, by a line bundle $\cL$ on $B$. Consider the prequantum $1$-shifted Lagrangian fibration
\[
\xymatrix{
X \ar[r] \ar[d] & X_{\dR} \ar[d] \\
B \ar[r] & B_{\dR}
}
\]
from \cref{ex:prequantumDR} specified, in particular, by the trivial gerbe on $B_{\dR}$ and the trivialization along $B\rightarrow B_{\dR}$ given by $\cL$. The geometric quantization of the Lagrangian $X\rightarrow X_{\dR}$ gives rise to an object $\cD_B\otimes_{\cO_B}\cL\in\cD(B)$ as well as to a functor $\cD(B)\rightarrow \QCoh(B)$ given by $\cE\mapsto \cE\otimes \cL$. We refer to \cite{ElliottYoo} for more details on this example.
\end{example}

We can also define the $0$-shifted version of geometric quantization.

\begin{defn}
Let $\pi\colon X\rightarrow B$ be a prequantum $0$-shifted Lagrangian fibration specified, in particular, by a line bundle $\cL$ on $B$. The \defterm{geometric quantization} of $\pi$ is the chain complex $\Gamma(B, \cL)$.
\end{defn}

\begin{example}
Let $X$ be a derived Artin stack locally of finite presentation and $\cL$ a line bundle on $X$. By \cref{thm:prequantumtwistedcotangent} the projection $\T^*_{c_1(\cL)} X\rightarrow X$ has a natural prequantization. Its geometric quantization is $\Gamma(X, \cL)$.
\end{example}

We may also describe geometric quantizations of a class of $0$-shifted Lagrangians.

\begin{defn}
Suppose $f\colon X\rightarrow Y$ is a morphism of derived stacks. $f$ has a \defterm{relative volume form of degree $n$} if for perfect complex $\cE$ on $Y$ there is a morphism
\[\int_f\colon \Gamma(X, f^*\cE)\longrightarrow \Gamma(Y, \cE)[-n].\]
\end{defn}

\begin{example}
Suppose $f\colon X\rightarrow Y$ is a smooth and proper schematic morphism of relative dimension $n$. Let $K_{X/Y}$ be the relative canonical bundle. Then for every perfect complex $\cE$ on $Y$ the Grothendieck duality provides an integration map
\[f_*(f^*\cE\otimes K_{X/Y})\longrightarrow \cE[-n].\]
Then a trivialization of the relative canonical bundle provides a relative volume form of degree $n$ in the above sense.
\end{example}

\begin{defn}
Consider a prequantum $0$-shifted Lagrangian correspondence
\[
\xymatrix@R=0cm{
& L \ar[dl] \ar[dr] \ar[dd] & \\
X \ar[dd] && Y \ar[dd] \\
& B_L \ar^{f}[dl] \ar_{g}[dr] & \\
B_X && B_Y
}
\]
specified, in particular, by line bundles $\cL_X$ on $B_X$ and $\cL_Y$ on $B_Y$ together with an isomorphism of their pullbacks $f^* \cL_X\cong g^* \cL_Y$. Moreover, assume $B_L\rightarrow B_Y$ has a relative volume form of degree $n$. Its \defterm{geometric quantization} is the morphism
\[\Gamma(B_X, \cL_X)\longrightarrow \Gamma(B_L, f^*\cL_X)\cong \Gamma(B_L, g^*\cL_Y)\longrightarrow \Gamma(B_Y, \cL_Y)[-n].\]
\end{defn}

Let us also mention a result about compositions of geometric quantizations.

\begin{prop}
Consider the diagram
\[
\xymatrix@R=0cm{
L_1 \ar[dr] \ar[dd] && L_2 \ar[dl] \ar[dd] \\
& X \ar[dd] & \\
B_{L_1} \ar_{f}[dr] && B_{L_2} \ar^{g}[dl] \\
& B_X &
}
\]
where both squares are prequantum $1$-shifted Lagrangian fibrations, where $B_{L_1}\rightarrow B_X$ is representable quasi-compact and quasi-separated. Let
\[F_{L_1}\colon \Ch\longrightarrow \QCoh^{\cG}(B_X),\qquad F_{L_2}\colon \QCoh^{\cG}(B_X)\longrightarrow \Ch\]
be the geometric quantizations of each square. Then the geometric quantization of the induced $0$-shifted Lagrangian fibration $L_1\times_X L_2\rightarrow B_{L_1}\times_{B_X} B_{L_2}$ is equivalent to the composite $F_{L_2}(F_{L_1}(k))$.
\label{prop:quantizationfunctor}
\end{prop}
\begin{proof}
By definition $F_{L_1}(k) = f_*\cO_{B_{L_1}}$ and $F_{L_2}(\cE) = \Gamma(B_{L_2}, g^*\cE)$. Therefore, the composite is
\[F_{L_2}(F_{L_1}(k)) = \Gamma(B_{L_2}, g^*f_*\cO_{B_{L_1}}).\]
Consider the Cartesian diagram
\[
\xymatrix{
B_{L_1}\times_{B_X} B_{L_2} \ar^{\tilde{g}}[r] \ar^{\tilde{f}}[d] & B_{L_1} \ar^{f}[d] \\
B_{L_2} \ar^{g}[r] & B_X
}
\]
By the base change formula (see \cite[Proposition 3.10]{BZFN} and \cite[Proposition 6.3.4.1]{SAG})
\[g^* f_*\cO_{B_{L_1}}\cong \tilde{f}_*\cO_{B_{L_1}\times_{B_X} B_{L_2}}.\]
\end{proof}

\begin{remark}
We expect that one can construct a symmetric monoidal $\infty$-category $\LagrCorr^{\prequant}_1$ which has the following informal description:
\begin{itemize}
\item Its objects are prequantum $1$-shifted Lagrangian fibrations $\pi_X\colon X\rightarrow B_X$.
\item Its morphisms from $(\pi_X\colon X\rightarrow B_X)$ to $(\pi_Y\colon Y\rightarrow B_Y)$ are prequantum $1$-shifted Lagrangian correspondences
\[
\xymatrix@R=0cm{
& L \ar[dd] \ar[dl] \ar[dr] & \\
X \ar[dd] && Y \ar[dd] \\
& B_L \ar[dl] \ar[dr] & \\
B_X && B_Y
}
\]
\end{itemize}
where $B_L\rightarrow B_Y$ is representable, quasi-compact and quasi-separated. We refer to \cite[Section 14]{HaugsengSpans} where such an $\infty$-category has been constructed without the prequantum data. We further expect that the geometric quantization promotes to a functor
\[\LagrCorr_1^{\prequant}\longrightarrow \PrL_k\]
given by \cref{def:geometricquantization1stack} on objects and by \cref{def:geometricquantization1correspondence} on morphisms. \Cref{prop:quantizationfunctor} is a shadow of the full functoriality statement.
\end{remark}

\begin{remark}
The recent paper \cite{StefanichPrL} defines iteratively the notion of a presentable $(\infty, n)$-category. Let $A$ be a connective commutative dg algebra. Then one may also iteratively define the $\infty$-category $\mathrm{Lin}n\mathrm{Cat}_A$ of linear $(\infty, n)$-categories by setting $\mathrm{Lin}0\mathrm{Cat}_A = \Mod_A$ and $\mathrm{Lin}n\mathrm{Cat}_A$ to be the subcategory
\[\mathrm{Lin}n\mathrm{Cat}_A \subset \Mod_{\mathrm{Lin}(n-1)\mathrm{Cat}_A}(n\PrL)\]
on $\kappa_0$-compact objects, where $\kappa_0$ is the smallest large cardinal. We expect that the \'{e}tale descent statement of $\mathrm{LinCat}_A = \mathrm{Lin}1\mathrm{Cat}_A$ from \cite[Section D.4]{SAG} generalizes for every $n$. Then one may define the $\infty$-category of quasi-coherent sheaves of $(\infty, n)$-categories on a derived prestack $X$ to be
\[\mathrm{Shv}n\mathrm{Cat}(X) = \lim_{\Spec A\rightarrow X} \mathrm{Lin}n\mathrm{Cat}_A.\]
We refer to the upcoming work of Stefanich for more details. In particular, this will allow one to define geometric quantization of a prequantum $n$-shifted Lagrangian fibration $X\rightarrow B$ for any $n\geq 0$.
\label{rem:LinCat}
\end{remark}

\section{Examples}
\label{sect:examples}

In this section we give several examples of prequantizations and geometric quantizations.

\subsection{Symplectic groupoids}
\label{sect:symplecticgroupoids}

Let $X$ be a smooth scheme. Recall the following definition from \cite{WeinsteinGroupoids}.

\begin{defn}
A \defterm{symplectic groupoid} is a groupoid $\cG\rightrightarrows X$ whose source and target morphisms are smooth morphisms together with a multiplicative symplectic form on $\cG$ such that the unit section $X\rightarrow \cG$ is Lagrangian.
\end{defn}

It is well-known (see e.g. \cite[Proposition 3.32]{SafronovPoissonLie}) that the data of a symplectic groupoid $\cG\rightrightarrows X$ may be encoded in a 1-shifted symplectic structure on the quotient stack $[X/\cG]$ together with a Lagrangian structure on the projection $X\rightarrow [X/\cG]$. Let us describe prequantizations and Lagrangian fibrations on this morphism.

Let $B$ be another smooth scheme together with a groupoid $\cG_B\rightrightarrows B$. Consider a smooth morphism of groupoids
\[\pi\colon (\cG\rightrightarrows X)\rightarrow (\cG_B\rightrightarrows B).\]
It gives rise to a commutative diagram
\[
\xymatrix{
X \ar[r] \ar[d] & [X/\cG] \ar[d] \\
B \ar[r] & [B/\cG_B]
}
\]
Choose a prequantum $1$-shifted Lagrangian fibration structure on this diagram. In particular, there is a gerbe $\cE$ on $[B/\cG_B]$ together with a trivialization of its pullback to $B$. Geometric quantization of this diagram gives rise to functors
\[Q_X\colon \Ch\longrightarrow \QCoh^{\cE}([B/\cG_B]),\qquad F\colon \QCoh^{\cE}([B/\cG_B])\longrightarrow \Ch.\]

\begin{remark}
The symplectic groupoid $\cG\rightrightarrows X$ induces a unique Poisson structure on $X$ such that the source map $\cG\rightarrow X$ is Poisson. The two algebras $\End(Q_X)$ and $\End(F)$ are both approximations to a deformation quantization of $X$; see \cite{WeinsteinXu} for more details on this idea.
\end{remark}

Given the diagram
\[
\xymatrix@R=0cm{
X \ar[dr] \ar[dd] && X \ar[dl] \ar[dd] \\
& [X/\cG] \ar[dd] & \\
B \ar[dr] && B \ar[dl] \\
& [B/\cG_B]
}
\]
each square is a $1$-shifted Lagrangian fibration. Therefore, passing to pullbacks, by \cref{thm:prequantumLagrangianfibrationintersection} the projection $\cG\rightarrow \cG_B$ is a $0$-shifted Lagrangian fibration. By construction it is easily seen to be multiplicative. In particular, the induced line bundle $\cL$ on $\cG_B$ is a multiplicative line bundle, i.e. it gives a central extension of the groupoid $\cG_B$. Such multiplicative polarizations and prequantizations of symplectic groupoids were extensively studied in \cite{Hawkins}.

The relationship between the geometric quantization of the symplectic groupoid $\cG$ and that of the 1-shifted Lagrangian morphism $X\rightarrow [X/\cG]$ is given by the following statement.

\begin{prop}
Suppose the source and target maps $\cG_B\rightarrow B\times B$ define a quasi-compact and quasi-separated morphism. Then the geometric quantization of $\cG$ is equivalent to $F(Q_X)$.
\end{prop}
\begin{proof}
Consider the correspondence
\[
\xymatrix@R=0cm{
& B \ar_{p}[dl] \ar^{\pi}[dr] & \\
\pt && [B/\cG_B]
}
\]
By assumption the morphism $\pi\colon B\rightarrow [B/\cG_B]$ is schematic, quasi-compact and quasi-separated. Therefore, the claim follows from \cref{prop:quantizationfunctor}.
\end{proof}

\begin{remark}
The 1-shifted symplectic stack $[X/\cG]$ is the phase space of the Poisson $\sigma$-model into $X$ equipped with Poisson structure induced from the symplectic groupoid \cite{SchallerStrobl,CattaneoFelder}. The Lagrangian $X\rightarrow [X/\cG]$ defines a classical boundary condition (``classical mechanics'').
\end{remark}

\subsection{Hamiltonian spaces}
\label{sect:Hamiltonian}

Let $G$ be a reductive algebraic group. Then
\[[\g^*/G]\cong \T^*[1](\B G)\]
carries a natural 1-shifted symplectic structure (see \cite[Section 1.2.3]{CalaqueTFT} for details). Therefore, by \cref{prop:cotangentprequantization} the projection $[\g^*/G]\rightarrow \B G$ has the structure of a prequantum $1$-shifted Lagrangian fibration specified by the trivial gerbe on $\B G$.

\begin{lm}
The geometric quantization of the prequantum $1$-shifted Lagrangian fibration $[\g^*/G] \rightarrow \B G$ is the $\infty$-category $\Rep(G)$ of complexes of $G$-representations.
\end{lm}
\begin{proof}
Since the gerbe on $\B G$ is trivial, the geometric quantization is $\QCoh(\B G)\cong \Rep(G)$.
\end{proof}

Now suppose $(X, \omega)$ is a symplectic scheme equipped with Hamiltonian $G$-action with a moment map $\mu\colon X\rightarrow \g^*$. If we denote by $a\colon \g\rightarrow \Gamma(X, \T_X)$ the infinitesimal action map, the moment map equation is
\[\iota_{a(x)}\omega = \ddr \mu(x).\]
It is shown in \cite[Section 2.2.1]{CalaqueTFT} that this gives rise to a $1$-shifted Lagrangian morphism
\[[X/G]\longrightarrow [\g^*/G].\]

\begin{thm}
Let $X$ be a Hamiltonian $G$-scheme. Choose:
\begin{itemize}
\item A smooth $G$-equivariant morphism $\pi\colon X\rightarrow B$ to a smooth scheme $B$ with half-dimensional fibers.
\item A $G$-equivariant line bundle $\cL$ on $B$.
\item A $G$-equivariant connection $\nabla$ on $\pi^*\cL$ extending the natural fiberwise connection, such that:
\begin{enumerate}
\item For any $x\in\g$ the operator $\nabla_{a(x)} + \mu(x)$ coincides with the infinitesimal $\g$-action on $\cL$.
\item The curvature of $\nabla$ is the symplectic form $\omega$.
\end{enumerate}
\end{itemize}
Then the diagram
\begin{equation}
\xymatrix{
[X/G] \ar^{\pi}[d] \ar[r] & [\g^*/G] \ar[d] \\
[B/G] \ar[r] & \B G
}
\label{eq:prequantumHamiltonianspace}
\end{equation}
is a prequantum $1$-shifted Lagrangian fibration. If $B$ is quasi-compact and quasi-separated, the geometric geometric quantization gives the functors $\Ch\rightarrow \Rep(G)$ given by including the $G$-representation $\Gamma(B, \cL)$ and $\Rep(G)\rightarrow\Ch$ given by $V\mapsto (V\otimes \Gamma(B, \cL))^G$.
\label{thm:prequantumHamiltonianspace}
\end{thm}
\begin{proof}
A $G$-equivariant line bundle $\cL$ on $B$ descends to a line bundle $\cL_{[B/G]}$ to $[B/G]$. We will use it as a trivialization of the pullback to $[B/G]$ of the trivial gerbe on $\B G$.

Since $G$ is reductive, we may compute $\DR([X/G])$ using the Cartan model of equivariant cohomology which gives
\[\DR([X/G]) = (\DR(X)\otimes \Sym(\g^*[-2]))^G.\]
The internal differential on $\DR([X/G])$ is the sum $\d + \iota_{a(-)}$, where the first term is the internal differential on $\DR(X)$ and the second term is the equivariant differential. The de Rham differential on $\DR([X/G])$ is simply the de Rham differential $\ddr$ on $\DR(X)$. The first Chern class $c_1(\cL_{[B/G]})\in\cA^1([B/G], 1)$ is the equivariant Chern class (see \cite{BerlineVergne,Getzler}) represented by the pair $(c_1(\cL), \beta)$, where $c_1(\cL)$ is the usual first Chern class, a degree 1 one-form on $B$, and $\beta\in\cO(X)\otimes \g^*$ which together satisfy
\[\iota_{a(x)} c_1(\cL) = \d\beta(x).\]
To be more explicit, we may use the \v{C}ech cohomology to compute $\DR(X)$. Choose a Zariski cover $\{U_i\}$ of $B$ over which $\cL$ trivializes and let $\{g_{ij}\}$ be the transition functions. Then $c_1(\cL)$ is represented by the one-forms $\{\dlog(g_{ij})\}$ on the intersections. The $G$-equivariance data for $\cL$ gives rise to functions $\beta_i\colon U_i\rightarrow \g^*$ such that
\[g_{ij}^{-1} a(x).g_{ij} = \beta_i(x) - \beta_j(x).\]
The pullback of the Liouville one-form on $[\g^*/G]$ to $[X/G]$ under the moment map is represented by $\mu\in\cO(X)\otimes \g^*$.

The prequantization data involves a trivialization of $\pi^*c_1(\cL_{[B/G]}) + \mu$. Pulling back the cover $U_i$ of $B$ to a cover $V_i$ of $X$, such a trivialization is given by one-forms $A_i$ on $U_i$ which satisfy the following equations:
\begin{align*}
A_i - A_j &= g_{ij}^{-1}\ddr g_{ij} \\
\iota_{a(x)}A_i + \mu(x) &= \beta_i(x).
\end{align*}
The first equation precisely says that $\{A_i\}$ define a connection $\nabla$ on $\pi^* \cL$ and the second equation is the condition that $\nabla_{a(x)} + \mu(x)$ gives the action of $\g$ on the line bundle.

The nondegeneracy condition on the 1-shifted isotropic structure on $[X/G]\rightarrow [\g^*/G]$ boils down to the fact that $\omega$ is symplectic as explained in \cite{CalaqueCotangent}. The nondegeneracy condition on the 1-shifted isotropic fibration on \eqref{eq:prequantumHamiltonianspace} is that $X\rightarrow B$ is a $0$-shifted Lagrangian fibration. But since these are smooth schemes, this condition is equivalent to the condition that the fibers are half-dimensional.

Finally, the statement about the geometric quantization of this 1-shifted Lagrangian fibration is clear as the pushforward of $\cL$ along $[B/G]\rightarrow \B G$ is $\Gamma(B, \cL)$ viewed as a $G$-representation and the pushforward of $\cL\otimes V$ (where $V\in\Rep(G)$) along $[B/G]\rightarrow \pt$ is $(\Gamma(B, \cL)\otimes V)^G$.
\end{proof}

\begin{remark}
The above data on $X$ is exactly the $G$-equivariance data for the prequantization and polarization on $X$ as in \cite{GuilleminSternberg}.
\end{remark}

\begin{remark}
The above statement has the following physical interpretation. The 1-shifted symplectic stack $[\g^*/G]$ is the phase space of the 2d BF theory on the point. The Lagrangian $[X/G]\rightarrow [\g^*/G]$ represents a classical boundary theory specified by a Hamiltonian $G$-space. After quantization $\Rep(G)$ is the category of boundary conditions of the 2d BF theory and the quantization of the above boundary theory gives a boundary condition specified by the object $\Gamma(B, \cL)\in\Rep(G)$.
\end{remark}

\begin{example}
Consider the Hamiltonian $G$-space $\T^* G$ corresponding to the left $G$-action on itself. Then the projection $\T^* G\rightarrow G$ is a $G$-equivariant Lagrangian fibration. By \cref{prop:prequantumcotangent} it admits a prequantization given the the trivial line bundle on $G$. By \cref{thm:prequantumHamiltonianspace} we can construct a $1$-shifted Lagrangian fibration
\[
\xymatrix{
\g^* \ar[d] \ar[r] & [\g^*/G] \ar[d] \\
\pt \ar[r] & \B G
}
\]
Its geometric quantization gives the functors
\[\Ch\longrightarrow \Rep(G),\qquad \Rep(G)\longrightarrow \Ch,\]
where the first functor sends $k\in\Ch$ to the regular representation $\cO(G)$ and the second functor is the forgetful functor.
\end{example}

\begin{remark}
There is a natural symplectic groupoid structure on $\T^* G\rightrightarrows \g^*$. In particular, the previous example fits into the general theory of geometric quantizations of symplectic groupoids, see \cite[Section 6.3]{Hawkins}.
\end{remark}

\begin{example}
Consider the Hamiltonian $G$-space $\pt$. The corresponding $1$-shifted Lagrangian fibration is
\[
\xymatrix{
\B G \ar@{=}[d] \ar[r] & [\g^*/G] \ar[d] \\
\B G \ar@{=}[r] & \B G
}
\]
Its geometric quantization gives the functors
\[\Ch\longrightarrow \Rep(G),\qquad \Rep(G)\longrightarrow \Ch,\]
where the first functor sends $k\in\Ch$ to the trivial one-dimensional representation and the second functor takes $G$-invariants.
\label{ex:prequantuminvariants}
\end{example}

For a Hamiltonian $G$-scheme $X$ we have two 1-shifted Lagrangians in $[\g^*/G]$: $[X/G]$ and $\B G$. In particular, we may consider their intersection. The following was introduced in \cite{CalaqueTFT,SafronovCS}.

\begin{defn}
Let $X$ be a Hamiltonian $G$-scheme. The \defterm{derived Hamiltonian reduction} is the $0$-shifted symplectic derived stack
\[X\ham G = X/G\times_{[\g^*/G]} \B G.\]
\end{defn}

The following statement is a derived analog of the ``quantization commutes with reduction'' principle \cite{GuilleminSternberg}.

\begin{prop}
Suppose $X$ is a Hamiltonian $G$-scheme equipped with a $G$-equivariant prequantization as in \cref{thm:prequantumHamiltonianspace}. Moreover, assume $B$ is a quasi-compact and quasi-separated. Then there is a natural prequantization of the $0$-shifted Lagrangian fibration $X\ham G\rightarrow B/G$ whose geometric quantization is $\Gamma(B, \cL)^G$.
\label{prop:QR}
\end{prop}
\begin{proof}
Consider the diagram
\[
\xymatrix@R=0cm{
[X/G] \ar[dr] \ar[dd] && \B G \ar@{=}[dd] \ar[dl] \\
& [\g^*/G] \ar[dd] & \\
[B/G] \ar[dr] && \B G \ar@{=}[dl] \\
& \B G
}
\]
The square on the left is a prequantum $1$-shifted Lagrangian fibration by \cref{thm:prequantumHamiltonianspace}. The square on the right is a prequantum $1$-shifted Lagrangian fibration by \cref{ex:prequantuminvariants}. Therefore, their intersection $X\ham G\rightarrow [B/G]$ carries a prequantum $0$-shifted Lagrangian fibration by \cref{thm:prequantumLagrangianfibrationintersection}. The geometric quantization of the square on the left gives $\Gamma(B, \cL)$ as a $G$-representation and the geometric quantization of the square on the right gives the functor of $G$-invariants. The claim then follows from \cref{prop:quantizationfunctor}.
\end{proof}

\subsection{Coadjoint orbits}
\label{sect:coadjoint}

Coadjoint orbits $\cO\subset \g^*$ are interesting examples of Hamiltonian $G$-spaces. Let us explain how prequantizations work for those examples. Assume that $G$ is split.

\begin{prop}
Let $x\in\g^*$ be a semisimple element with stabilizer $L$ with Lie algebra $\Lie(L) = \fl$ under the coadjoint action. Let $P\subset G$ be a parabolic subgroup with Lie algebra $\Lie(P) = \fp$ containing $L$ as the Levi factor. Assume that $x$ integrates to a character $\lambda\colon P\rightarrow L\rightarrow \Gm$. Then
\[
\xymatrix{
\B L \ar^{x}[r] \ar[d] & [\g^*/G] \ar[d] \\
\B P \ar[r] & \B G
}
\]
is a prequantum $1$-shifted Lagrangian fibration. Suppose the restriction of $\lambda$ to the maximal torus $H\subset P$ is dominant. Then the geometric quantization is $V(\lambda)^*\in\Rep(G)$, the dual of the irreducible $G$-representation with highest weight $\lambda$.
\label{prop:semisimplecoadjoint}
\end{prop}
\begin{proof}
To specify the prequantum data, we first have to trivialize the trivial gerbe on $\B P$. Such a trivialization is given by a line bundle $\cL$ on $\B P$, i.e. a character $\lambda\colon P\rightarrow \Gm$. Let us also recall that a character of $P$ necessarily factors through a character of $L$.

Next, we have to trivialize the connective structure on the gerbe on $\B L$. Note that $P / L$ is affine, so $\cA^1(\B L/\B P, 1)$ is contractible. Therefore, we just need to identify the pullback of the Liouville one-form to $\B L$ with the pullback of $\ddr \lambda$ in $\cA^1(\B L, 1) = (\fl^*)^L$, i.e. we have to assume that $x\in(\fl^*)^L$ is integral. This finishes the construction of the $1$-shifted isotropic structure on the above diagram.

Next, we have to check the relevant nondegeneracy conditions. To show that the isotropic structure on $\B L\rightarrow [\g^*/G]$ is Lagrangian, we have to check that the sequence of $L$-representations
\[\fl[1]\longrightarrow(\g[1]\xrightarrow{\coad_x} \g^*)\longrightarrow \fl^*\]
is a fiber sequence, where $\coad_x\colon \g\rightarrow \g^*$ denotes the coadjoint action on $x$. The previous fact is equivalent to acyclicity of the complex
\[\fl\longrightarrow \g\longrightarrow \g^*\longrightarrow \fl^*\]
which is equivalent to the fact that $\coad_x\colon \g/\fl\rightarrow (\g/\fl)^*$ is an isomorphism. This map is clearly injective and the two spaces have the same dimension which proves the claim.

The fact that the above diagram defines a $1$-shifted Lagrangian fibration on $\B L\rightarrow [\g^*/G]$ boils down to exactness of the sequence
\[\fp/\fl\longrightarrow \g/\fl\cong (\g/\fl)^*\longrightarrow (\fp/\fl)^*,\]
where the middle isomorphism is given by $\coad_x$. In turn, this fact is equivalent to
\[\coad_x\colon \g/\fp\longrightarrow (\fp/\fl)^*\]
being an isomorphism. By checking the dimensions it is enough to prove that this map is injective. To prove it, we may assume that $P\subset G$ is a standard parabolic subgroup \cite[Proposition 14.18]{Borel}. Decompose $\fp = \fl \oplus \fu^+$ and let $\fl\oplus \fu^-$ be the opposite parabolic subalgebra.

Choose a nondegenerate invariant symmetric bilinear pairing $\langle -, -\rangle$ on $\g$. For every standard $\sl_2$ triple $\{e_\alpha, e_{-\alpha}, h_\alpha\}$ we have
\[\langle x, [e_\alpha, e_{-\alpha}]\rangle = -\langle [e_\alpha, x], e_{-\alpha}\rangle.\]
Since the root spaces are one-dimensional, this expression vanishes only if $e_\alpha\in\fl$. Therefore, $\langle x, h_\alpha\rangle$ is nonzero for every root in $\fu^+$ and therefore the induced pairing $\fu^+\otimes \fu^-\rightarrow k$ is nondegenerate.

The geometric quantization is given by the pushforward of $\cL$ along $\B P\rightarrow \B G$. Let $B\subset P$ be the Borel subgroup. By \cite[Part II, Proposition 4.6]{Jantzen} the pushforward coincides with the pushforward along $\B B\rightarrow \B G$ of the restriction of $\cL$ to $\B B$ and the claim follows from the Borel--Weil--Bott theorem.
\end{proof}

\begin{remark}
Suppose $\cO\subset \g^*$ is a coadjoint orbit. A $G$-equivariant $0$-shifted Lagrangian fibration $\cO\rightarrow B$ has to have $B$ a homogeneous space $G/P$. It is shown in \cite{OzekiWakimoto} that such a $P$ has to be a parabolic subgroup. In particular, by counting dimensions one observes that there are many coadjoint orbits which do not admit a $G$-equivariant polarization, see e.g. \cite[Corollary 1.8]{GrahamVogan}.
\end{remark}

\subsection{Slodowy slices}
\label{sect:Slodowy}

Assume $k$ is an algebraically closed field of characteristic $0$ and $G$ is a semisimple algebraic group with Lie algebra $\g$ equipped with nondegenerate invariant symmetric bilinear pairing $\langle -, -\rangle$. We denote by $\kappa\colon \g\rightarrow \g^*$ the induced isomorphism. Let $e\in\g$ be a nonzero nilpotent element. Using the Jacobson--Morozov theorem \cite[Theorem 3.7.1]{ChrissGinzburg} we may extend it to an $\sl_2$-triple $\{e, h, f\}$. Denote
\[\chi = \langle e, -\rangle\]
and $\g^f\subset \g$ the subspace of elements $x\in\g$ such that $[f, x] = 0$.

\begin{defn}
The \defterm{Slodowy slice} to $\chi$ is the subvariety $\cS = \chi + \kappa(\g^f)\subset \g^*$.
\end{defn}

We refer to \cite{GanGinzburg,Arakawa} for more details on Slodowy slices.

\begin{remark}
If $e$ is a regular nilpotent element, the Slodowy slice reduces to the Kostant slice \cite{Kostant}.
\end{remark}

Let $\g=\oplus_{n\in\Z} \g(n)$ be the decomposition into the $h$-eigenspaces. We get an induced symplectic pairing $x,y\mapsto \chi([x, y])$ on $\g(-1)$. Fix a Lagrangian subalgebra $\fl\subset \g(-1)$ and let
\[\fm = \fl\oplus \bigoplus_{n\leq -2} \g(n).\]
It is a nilpotent subalgebra of $\g$ and we denote by $M\subset G$ the unipotent subgroup integrating $\fm$. It is easy to see that $\chi$ restricts to a character of $\fm$.

\begin{prop}
Consider the zero closed one-form of degree $2$ on $\B G$ together with a trivialization along $\B M\rightarrow \B G$ given by $\chi$. Then the $1$-shifted twisted conormal bundle
\[
\xymatrix{
\N^*_{\chi}[1](\B M) \ar[d] \ar[r] & \T^*[1](\B G) \ar[d] \\
\B M \ar[r] & \B G
}
\]
is equivalent to
\[
\xymatrix{
\cS \ar[r] \ar[d] & [\g^*/G] \ar[d] \\
\B M \ar[r] & \B G
}
\]
In particular, the latter diagram has a natural structure of a $1$-shifted Lagrangian fibration.
\label{prop:slodowy}
\end{prop}
\begin{proof}
The conormal bundle to $\B M\rightarrow \B G$ fits into a fiber sequence
\[\N^*_{\B M}\rightarrow \g^*[-1]\xrightarrow{\mu} \fm^*[-1]\]
of $M$-representations (we recall that $\QCoh(\B M)\cong \Rep(M)$). Therefore,
\[(\N^*_{\chi}[1](\B M)\rightarrow \B M)\cong (\mu^{-1}(\chi) / M\rightarrow \B M).\]
But the space $\mu^{-1}(\chi) / M$ is isomorphic to the Slodowy slice $\cS$, see \cite[Lemma 2.1]{GanGinzburg}.
\end{proof}

\begin{remark}
The product $\cS\times G$ is a Hamiltonian $G$-space which may be identified with the space of solutions of Nahm's equations on the interval $[0, 1]$ with a pole at $t = 0$ specified by the $\sl_2$-triple $\{e, h, f\}$, see \cite[Proposition 3.1]{Bielawski}.
\end{remark}

\begin{remark}
The diagram in \cref{prop:slodowy} does \emph{not} admit a prequantization since $\chi\colon \fm\rightarrow k$ does not integrate to a character of $M$. Instead, one can work with the formal completions $\B \widehat{M}$ and $\B \widehat{G}$. In this case the geometric quantization of the diagram in \cref{prop:slodowy} produces the functor of Whittaker $(\fm, \chi)$-invariants $\LMod_{\mathrm{U}\g}\rightarrow \Ch$.
\end{remark}

\subsection{Flat connections}
\label{sect:CS}

Let $G$ be a split connected simply-connected semisimple group with Lie algebra $\g$. Let $H\subset G$ be a maximal torus with Lie algebra $\h$. Denote by $W$ the Weyl group. Recall that a 2-shifted symplectic structure $\omega$ on $\B G$ is the same as a nondegenerate invariant symmetric bilinear pairing $\langle -, -\rangle\in\Sym^2(\g^*)^G\cong \Sym^2(\h^*)^W$. We say it is \defterm{integral} if it comes from an integral quadratic Weyl-invariant form on the cocharacter lattice.

Let $C$ be a smooth and proper curve. Let
\[\LocSys_G(C) = \Map(C_{\dR}, \B G)\]
be the moduli stack of flat $G$-connections on $C$ and
\[\Bun_G(C) = \Map(C, \B G)\]
be the moduli stack of $G$-bundles on $C$. Recall from \cref{sect:deRham} that there is a natural closed one-form of degree $1$ $\int_C \ev^*\omega$ on $\Bun_G(C)$. According to \cref{prop:deRhamMap} there is an isomorphism
\[\LocSys_G(C)\cong \T^*_{\int_C \ev^*\omega} \Bun_G(C).\]
In particular, we may consider the corresponding $0$-shifted symplectic structure.

\begin{prop}
Suppose the pairing $\langle -, -\rangle$ is integral. Then there is a natural line bundle $\cL$ on $\Bun_G(C)$ such that $c_1(\cL)=\int_C \ev^*\omega$.
\label{prop:Gaitsgory}
\end{prop}
\begin{proof}
The line bundle is constructed in \cite[Section 2.4]{GaitsgoryFactorizable} using the Gersten resolution for $\cK$-cohomology on the product $S\times C$, where $S$ is a smooth variety. There is a natural map $\dlog$ from the Gersten complex to the Cousin complex computing $\H^q(S\times C, \Omega^{p, \cl})$, see e.g. \cite[Proposition 2.1.9]{BraunlingWolfson}. So, replacing the Gersten complex with the Cousin complex in \cite{GaitsgoryFactorizable}, we obtain a construction of the closed one-form of degree 1 on $\Bun_G(C)$ representing $c_1(\cL)$. Finally, the equivalence between that representative of $c_1(\cL)$ and $\int_C \ev^*\omega$ is given by using the formulation of Serre duality in terms of the Cousin complex \cite{Hartshorne}.
\end{proof}

\begin{thm}
Suppose the pairing $\langle -, -\rangle$ is integral. Then there is a natural prequantum $0$-shifted Lagrangian fibration structure on $\LocSys_G(C)\rightarrow \Bun_G(C)$ specified by a line bundle $\cL$ on $\Bun_G(C)$. Its geometric quantization is
\[\Gamma(\Bun_G(C), \cL).\]
\label{thm:ChernSimons}
\end{thm}
\begin{proof}
By \cref{prop:Gaitsgory} there is a line bundle $\cL$ on $\Bun_G(C)$, such that $c_1(\cL) = \int_C\ev^*\omega$. In particular, by \cref{prop:deRhamMap} we obtain an isomorphism
\[\LocSys_G(C)\cong \T^*_{c_1(\cL)}\Bun_G(C).\]
The claim then follows from \cref{thm:prequantumtwistedcotangent}.
\end{proof}

\begin{remark}
Let $K$ be a compact simply-connected simple Lie group and let $G$ be its complexification. In this case $\Sym^2(\g^*)^G\cong \C$, so a pairing is specified by a number (the \emph{level}). Integral pairings correspond to integral levels. The phase space of the classical Chern--Simons theory on $C$ is the moduli space of flat $K$-connections. The Hilbert space on $C$ is given by $\Gamma(\Bun_G(C), \cL)$ \cite{WittenCS,ADPW} and \cref{thm:ChernSimons} provides an algebraic version of this result.
\end{remark}

Denote by $\LocSys_G(C, \trivial)$ the moduli space of flat connections on the trivial $G$-bundle over $C$.

\begin{prop}
Suppose the pairing $\langle -, -\rangle$ is integral. The diagram
\begin{equation}
\xymatrix{
\LocSys_G(C, \trivial) \ar[d] \ar[r] & \LocSys_G(C) \ar[d] \\
\pt \ar^{\trivial}[r] & \Bun_G(C)
}
\label{eq:chiralWZW}
\end{equation}
has a natural structure of a prequantum $0$-shifted Lagrangian fibration. Its geometric quantization is the functional
\[\Gamma(\Bun_G(C), \cL)\longrightarrow k\]
given by evaluating the section of $\cL$ at the trivial bundle.
\label{prop:chiralWZW}
\end{prop}
\begin{proof}
Consider the inclusion of the trivial bundle $\trivial\colon \pt\rightarrow \Bun_G(C)$. The pullback of $c_1(\cL)$ to $\pt$ is canonically trivialized, so we may consider the corresponding twisted conormal bundle
\[\N^*_{c_1(\cL)}(\pt)\longrightarrow \T^*_{c_1(\cL)}\Bun_G(C).\]
By definition of $\LocSys_G(C, \trivial)$, the diagram \eqref{eq:chiralWZW} is Cartesian. In particular,
\[\N^*_{c_1(\cL)}(\pt)\cong \LocSys_G(C, \trivial).\]
Choose a trivialization of $\cL$ at the trivial bundle $\pt\rightarrow \Bun_G(C)$. Then we obtain a prequantum data on the diagram.
\end{proof}

\begin{remark}
The chiral WZW model is a boundary condition for the Chern--Simons theory. Namely, it is determined by a Lagrangian in the phase space given by setting the antiholomorphic part of the connection to be zero (see e.g. \cite{WittenFactorization}). In other words, the space of classical solutions of the chiral WZW model on $C$ is given by studying holomorphic connections on the trivial holomorphic $G$-bundle over $C$. The chiral WZW partition function on $C$ is an element of (or a functional on) the Hilbert space of the Chern--Simons theory associated to $C$ and \cref{prop:chiralWZW} provides an algebraic avatar of this computation.
\end{remark}

\printbibliography

\end{document}